\documentclass[a4paper,11pt]{amsart}    
\usepackage{latexsym, amsmath, amsthm, amssymb, setspace, verbatim}
\usepackage[all]{xy}
\usepackage[utf8]{inputenc} \usepackage[T1]{fontenc} \usepackage{lmodern}
\usepackage{hyperref, aliascnt}
\usepackage{color}
\usepackage[usenames,dvipsnames]{xcolor}
\usepackage{enumitem}
\usepackage{ulem}

\title{Rokhlin dimension: absorption of model actions}
\author{Gábor Szabó}
\address{Department of Mathematics, KU Leuven, Celestijnenlaan 200b, box 2400\linebreak
\phantom{-}\hspace{2mm} B-3001 Leuven, Belgium}
\email{gabor.szabo@kuleuven.be}
\thanks{This work was supported by the Danish National Research Foundation through the Centre for Symmetry and Deformation (DNRF92), and the European Union's Horizon 2020 research and innovation programme under the Marie Sklodowska-Curie grant agreement 746272.}
\subjclass[2010]{46L55}

\numberwithin{equation}{section}

\begin{document}

\renewcommand\matrix[1]{\left(\begin{array}{*{10}{c}} #1 \end{array}\right)}  
\newcommand\set[1]{\left\{#1\right\}}  

\newcommand{\IA}[0]{\mathbb{A}} \newcommand{\IB}[0]{\mathbb{B}}
\newcommand{\IC}[0]{\mathbb{C}} \newcommand{\ID}[0]{\mathbb{D}}
\newcommand{\IE}[0]{\mathbb{E}} \newcommand{\IF}[0]{\mathbb{F}}
\newcommand{\IG}[0]{\mathbb{G}} \newcommand{\IH}[0]{\mathbb{H}}
\newcommand{\II}[0]{\mathbb{I}} \renewcommand{\IJ}[0]{\mathbb{J}}
\newcommand{\IK}[0]{\mathbb{K}} \newcommand{\IL}[0]{\mathbb{L}}
\newcommand{\IM}[0]{\mathbb{M}} \newcommand{\IN}[0]{\mathbb{N}}
\newcommand{\IO}[0]{\mathbb{O}} \newcommand{\IP}[0]{\mathbb{P}}
\newcommand{\IQ}[0]{\mathbb{Q}} \newcommand{\IR}[0]{\mathbb{R}}
\newcommand{\IS}[0]{\mathbb{S}} \newcommand{\IT}[0]{\mathbb{T}}
\newcommand{\IU}[0]{\mathbb{U}} \newcommand{\IV}[0]{\mathbb{V}}
\newcommand{\IW}[0]{\mathbb{W}} \newcommand{\IX}[0]{\mathbb{X}}
\newcommand{\IY}[0]{\mathbb{Y}} \newcommand{\IZ}[0]{\mathbb{Z}}

\newcommand{\CA}[0]{\mathcal{A}} \newcommand{\CB}[0]{\mathcal{B}}
\newcommand{\CC}[0]{\mathcal{C}} \newcommand{\CD}[0]{\mathcal{D}}
\newcommand{\CE}[0]{\mathcal{E}} \newcommand{\CF}[0]{\mathcal{F}}
\newcommand{\CG}[0]{\mathcal{G}} \newcommand{\CH}[0]{\mathcal{H}}
\newcommand{\CI}[0]{\mathcal{I}} \newcommand{\CJ}[0]{\mathcal{J}}
\newcommand{\CK}[0]{\mathcal{K}} \newcommand{\CL}[0]{\mathcal{L}}
\newcommand{\CM}[0]{\mathcal{M}} \newcommand{\CN}[0]{\mathcal{N}}
\newcommand{\CO}[0]{\mathcal{O}} \newcommand{\CP}[0]{\mathcal{P}}
\newcommand{\CQ}[0]{\mathcal{Q}} \newcommand{\CR}[0]{\mathcal{R}}
\newcommand{\CS}[0]{\mathcal{S}} \newcommand{\CT}[0]{\mathcal{T}}
\newcommand{\CU}[0]{\mathcal{U}} \newcommand{\CV}[0]{\mathcal{V}}
\newcommand{\CW}[0]{\mathcal{W}} \newcommand{\CX}[0]{\mathcal{X}}
\newcommand{\CY}[0]{\mathcal{Y}} \newcommand{\CZ}[0]{\mathcal{Z}}

\newcommand{\FA}[0]{\mathfrak{A}} \newcommand{\FB}[0]{\mathfrak{B}}
\newcommand{\FC}[0]{\mathfrak{C}} \newcommand{\FD}[0]{\mathfrak{D}}
\newcommand{\FE}[0]{\mathfrak{E}} \newcommand{\FF}[0]{\mathfrak{F}}
\newcommand{\FG}[0]{\mathfrak{G}} \newcommand{\FH}[0]{\mathfrak{H}}
\newcommand{\FI}[0]{\mathfrak{I}} \newcommand{\FJ}[0]{\mathfrak{J}}
\newcommand{\FK}[0]{\mathfrak{K}} \newcommand{\FL}[0]{\mathfrak{L}}
\newcommand{\FM}[0]{\mathfrak{M}} \newcommand{\FN}[0]{\mathfrak{N}}
\newcommand{\FO}[0]{\mathfrak{O}} \newcommand{\FP}[0]{\mathfrak{P}}
\newcommand{\FQ}[0]{\mathfrak{Q}} \newcommand{\FR}[0]{\mathfrak{R}}
\newcommand{\FS}[0]{\mathfrak{S}} \newcommand{\FT}[0]{\mathfrak{T}}
\newcommand{\FU}[0]{\mathfrak{U}} \newcommand{\FV}[0]{\mathfrak{V}}
\newcommand{\FW}[0]{\mathfrak{W}} \newcommand{\FX}[0]{\mathfrak{X}}
\newcommand{\FY}[0]{\mathfrak{Y}} \newcommand{\FZ}[0]{\mathfrak{Z}}

\newcommand{\fc}[0]{\mathfrak{c}}
\newcommand{\fp}[0]{\mathfrak{p}}
\newcommand{\fn}[0]{\mathfrak{n}}
\newcommand{\mfi}[0]{\mathfrak{i}}
\newcommand{\mfj}[0]{\mathfrak{j}}

\newcommand{\Ra}[0]{\Rightarrow}
\newcommand{\La}[0]{\Leftarrow}
\newcommand{\LRa}[0]{\Leftrightarrow}

\renewcommand{\phi}[0]{\varphi}
\newcommand{\eps}[0]{\varepsilon}

\newcommand{\ord}[0]{\operatorname{ord}}		
\newcommand{\GL}[0]{\operatorname{GL}}
\newcommand{\supp}[0]{\operatorname{supp}}	
\newcommand{\id}[0]{\operatorname{id}}		
\newcommand{\eins}[0]{\mathbf{1}}			
\newcommand{\diag}[0]{\operatorname{diag}}
\newcommand{\prim}[0]{\operatorname{Prim}}
\newcommand{\ad}[0]{\operatorname{Ad}}
\newcommand{\ext}[0]{\operatorname{Ext}}
\newcommand{\ev}[0]{\operatorname{ev}}
\newcommand{\fin}[0]{{\subset\!\!\!\subset}}
\newcommand{\diam}[0]{\operatorname{diam}}
\newcommand{\Hom}[0]{\operatorname{Hom}}
\newcommand{\Aut}[0]{\operatorname{Aut}}
\newcommand{\del}[0]{\partial}
\newcommand{\dimnuc}[0]{\dim_{\mathrm{nuc}}}
\newcommand{\dr}[0]{\operatorname{dr}}
\newcommand{\dimrok}[0]{\dim_{\mathrm{Rok}}}
\newcommand{\dimrokc}[0]{\dim_{\mathrm{Rok}}^{\mathrm{c}}}
\newcommand{\asdim}[0]{\operatorname{asdim}}
\newcommand{\dimnuceins}[0]{\dimnuc^{\!+1}}
\newcommand{\dreins}[0]{\dr^{\!+1}}
\newcommand{\dimrokeins}[0]{\dimrok^{\!+1}}
\newcommand{\mdim}[0]{\operatorname{mdim}}
\newcommand*\onto{\ensuremath{\joinrel\relbar\joinrel\twoheadrightarrow}} 
\newcommand*\into{\ensuremath{\lhook\joinrel\relbar\joinrel\rightarrow}}  
\newcommand{\dst}[0]{\displaystyle}
\newcommand{\cstar}[0]{$\mathrm{C}^*$}
\newcommand{\dist}[0]{\operatorname{dist}}
\newcommand{\ue}[1]{{~\approx_{\mathrm{u},#1}}~}
\newcommand{\End}[0]{\operatorname{End}}
\newcommand{\Ell}[0]{\operatorname{Ell}}
\newcommand{\ann}[0]{\operatorname{Ann}}
\newcommand{\strict}[0]{\stackrel{\text{\tiny str}}{\longrightarrow}}
\newcommand{\cc}[0]{\simeq_{\mathrm{cc}}}
\newcommand{\scc}[0]{\simeq_{\mathrm{scc}}}
\newcommand{\vscc}[0]{\simeq_{\mathrm{vscc}}}
\newcommand{\inter}[1]{\ensuremath{#1^{\mathrm{o}}}}

\newcommand{\greater}[0]{>}
\newcommand{\vslash}[0]{|}
\newcommand{\norm}[0]{\ensuremath{\|}}
\newcommand{\msout}[1]{\text{\sout{$#1$}}}

\newcommand{\cf}[2]{cf.\ {\cite[#2]{#1}}}
\renewcommand{\see}[2]{see {\cite[#2]{#1}}}

\renewcommand{\emph}[1]{{\it #1}}

\newtheorem{satz}{Satz}[section]		

\newaliascnt{corCT}{satz}
\newtheorem{cor}[corCT]{Corollary}
\aliascntresetthe{corCT}
\providecommand*{\corCTautorefname}{Corollary}
\newaliascnt{lemmaCT}{satz}
\newtheorem{lemma}[lemmaCT]{Lemma}
\aliascntresetthe{lemmaCT}
\providecommand*{\lemmaCTautorefname}{Lemma}
\newaliascnt{propCT}{satz}
\newtheorem{prop}[propCT]{Proposition}
\aliascntresetthe{propCT}
\providecommand*{\propCTautorefname}{Proposition}
\newaliascnt{theoremCT}{satz}
\newtheorem{theorem}[theoremCT]{Theorem}
\aliascntresetthe{theoremCT}
\providecommand*{\theoremCTautorefname}{Theorem}
\newtheorem*{theoreme}{Theorem}
\newtheorem*{core}{Corollary}

\theoremstyle{definition}

\newaliascnt{conjectureCT}{satz}
\newtheorem{conjecture}[conjectureCT]{Conjecture}
\aliascntresetthe{conjectureCT}
\providecommand*{\conjectureCTautorefname}{Conjecture}
\newaliascnt{defiCT}{satz}
\newtheorem{defi}[defiCT]{Definition}
\aliascntresetthe{defiCT}
\providecommand*{\defiCTautorefname}{Definition}
\newtheorem*{defie}{Definition}
\newaliascnt{notaCT}{satz}
\newtheorem{nota}[notaCT]{Notation}
\aliascntresetthe{notaCT}
\providecommand*{\notaCTautorefname}{Notation}
\newtheorem*{notae}{Notation}
\newaliascnt{remCT}{satz}
\newtheorem{rem}[remCT]{Remark}
\aliascntresetthe{remCT}
\providecommand*{\remCTautorefname}{Remark}
\newtheorem*{reme}{Remark}
\newaliascnt{exampleCT}{satz}
\newtheorem{example}[exampleCT]{Example}
\aliascntresetthe{exampleCT}
\providecommand*{\exampleCTautorefname}{Example}
\newaliascnt{questionCT}{satz}
\newtheorem{question}[questionCT]{Question}
\aliascntresetthe{questionCT}
\providecommand*{\questionCTautorefname}{Question}

\newcounter{theoremintro}
\renewcommand*{\thetheoremintro}{\Alph{theoremintro}}
\newaliascnt{theoremiCT}{theoremintro}
\newtheorem{theoremi}[theoremiCT]{Theorem}
\aliascntresetthe{theoremiCT}
\providecommand*{\theoremiCTautorefname}{Theorem}
\newaliascnt{coriCT}{theoremintro}
\newtheorem{cori}[coriCT]{Corollary}
\aliascntresetthe{coriCT}
\providecommand*{\coriCTautorefname}{Corollary}


\begin{abstract} 
In this paper, we establish a connection between Rokhlin dimension and the absorption of certain model actions on strongly self-absorbing \cstar-algebras.
Namely, as to be made precise in the paper, let $G$ be a well-behaved locally compact group.
If $\CD$ is a strongly self-absorbing \cstar-algebra, and $\alpha: G\curvearrowright A$ is an action on a separable, $\CD$-absorbing \cstar-algebra that has finite Rokhlin dimension with commuting towers, then $\alpha$ tensorially absorbs every semi-strongly self-absorbing $G$-actions on $\CD$.
In particular, this is the case when $\alpha$ satisfies any version of what is called the Rokhlin property, such as for $G=\IR$ or $G=\IZ^k$.
This contains several existing results of similar nature as special cases.
We will in fact prove a more general version of this theorem, which is intended for use in subsequent work.
We will then discuss some non-trivial applications.
Most notably it is shown that for any $k\geq 1$ and on any strongly self-absorbing Kirchberg algebra, there exists a unique $\IR^k$-action having finite Rokhlin dimension with commuting towers up to (very strong) cocycle conjugacy.
\end{abstract}

\maketitle

\tableofcontents


\section*{Introduction}

\noindent
The present work is a continuation of the author's quest to study fine structure and classification of certain \cstar-dynamics by employing ideas related to tensorial absorption.
In previous work, the theory of (semi-)strongly self-absorbing actions on \cstar-algebras \cite{Szabo18ssa, Szabo18ssa2, Szabo17ssa3} was developed, closely following the important results established in the classical theory of strongly self-absorbing \cstar-algebras by Toms--Winter and others \cite{TomsWinter07, Kirchberg04, DadarlatWinter09}.
Strongly self-absorbing \cstar-algebras have historically emerged by example \cite{JiangSu99}, and by now play a central role in the structure theory of simple nuclear \cstar-algebras; see for example \cite{KirchbergPhillips00, Rordam04, ElliottToms08, Winter10, WinterZacharias10, Winter12, Winter14Lin, MatuiSato12, MatuiSato14UHF, BBSTWW, CETWW}.
Roughly speaking, a tensorial factorization of the form $A\cong A\otimes\CD$ --- for a given \cstar-algebra $A$ and a strongly self-absorbing \cstar-algebra $\CD$ --- provides sufficient space to perform non-trival manipulations on elements inside $A$, which often gives rise to structural properties of particular interest for classification.
The underlying motivation behind \cite{Szabo18ssa, Szabo18ssa2, Szabo17ssa3} is the idea that this kind of phenomenon should persist at the level of \cstar-dynamics if one is interested in classification of group actions up to cocycle conjugacy; in fact some much earlier work \cite{Kishimoto01, Kishimoto02, IzumiMatui10, GoldsteinIzumi11, MatuiSato12_2, MatuiSato14, Izumi12OWR} has (sometimes implicitly) used this idea to reasonable success.
It was further demonstrated in \cite{Szabo17ssa3, Szabo18kp} how this approach can indeed give rise to new insights about classification or rigidity of group actions on certain \cstar-algebras, in particular strongly self-absorbing ones.

Starting from Connes' groundbreaking work \cite{Connes75, Connes76, Connes77} on injective factors, which involved classification of single automorphisms, the Rokhlin property in its various forms became a key tool to classify actions of amenable groups on von Neumann algebras \cite{Jones80, Ocneanu85, SutherlandTakesaki89, KawahigashiSutherlandTakesaki92, KatayamaSutherlandTakesaki98, Masuda07}.
It did not take long for these ideas to reach the realm of \cstar-algebras.
Initially appearing in works of Herman--Jones \cite{HermanJones82} and Herman--Ocneanu \cite{HermanOcneanu84}, the Rokhlin property for single automorphisms and its applications for classification were perfected in works of Kishimoto and various collaborators \cite{BratteliKishimotoRordamStormer93, Kishimoto95, BratteliEvansKishimoto95, Kishimoto96, EvansKishimoto97, Kishimoto98, Kishimoto98II, ElliottEvansKishimoto98, BratteliKishimoto00, Nakamura00}.
Further work pushed these techniques to actions of infinite higher-rank groups as well \cite{Nakamura99, KatsuraMatui08, Matui08, Matui10, Matui11, IzumiMatui10, Izumi12OWR}.
The case of finite groups was treated in work of Izumi \cite{Izumi04, Izumi04II}, where it was shown that such actions with the Rokhlin property have a particularly rigid theory; see also \cite{Santiago15, GardellaSantiago16, Gardella14_1, Gardella14_2, Gardella17, BarlakSzabo16ss, BarlakSzaboVoigt16}.
In contrast to von Neumann algebras, however, the Rokhlin property for actions on \cstar-algebras has too many obstructions in general, ranging from obvious ones like lack of projections to more subtle ones of $K$-theoretic nature.

Rokhlin dimension is a notion of dimension for actions of certain groups on \cstar-algebras and was first introduced by Hirshberg--Winter--Zacharias \cite{HirshbergWinterZacharias15}. 
Several natural variants of Rokhlin dimension have been introduced, and all of them have in common that they generalize (to some degree) the Rokhlin property for actions of either finite groups or the integers. 
The theory has been extended and applied in many following works such as \cite{Szabo15plms, HirshbergPhillips15, SzaboWuZacharias17, Gardella17, HirshbergSzaboWinterWu17, Liao16, Liao17, BrownTikuisisZelenberg17, GardellaKalantarLupini17}.
In short, the advantage of working with Rokhlin dimension is that it is both more prevalent and more flexible than the Rokhlin property, but is yet often strong enough to deduce interesting structural properties of the crossed product, such as finite nuclear dimension \cite{WinterZacharias10}. 

A somewhat stronger version of Rokhlin dimension, namely with commuting towers, has been considered from the very beginning as a variant that was also compatible with respect to the absorption of strongly self-absorbing \cstar-algebras.
Although the assumption of commuting towers initially only looked like a minor technical assumption, it was eventually discovered that it can make a major difference in some cases such as for actions of finite groups \cite{HirshbergPhillips15}.

The purpose of this paper is to showcase a decisive connection between finite Rokhlin dimension with commuting towers and the absorption of semi-strongly self-absorbing model actions.
The following describes a variant of the main result; see \autoref{thm:dimrokc-main}:

\begin{theoremi} \label{Thm-A}
Let $G$ be a second-countable, locally compact group and $N\subset G$ a closed, normal subgroup.
Suppose that the quotient $G/N$ contains a discrete, normal, cocompact subgroup that is residually finite and has a box space with finite asymptotic dimension.
Let $A$ be a separable \cstar-algebra with an action $\alpha: G\curvearrowright A$.
Let $\gamma: G\curvearrowright\CD$ be a semi-strongly self-absorbing action that is unitarily regular.
Suppose that $\alpha\vslash_N$ is $\gamma\vslash_N$-absorbing.
If the Rokhlin dimension of $\alpha$ with commuting towers relative to $N$ is finite, then it follows that $\alpha$ is $\gamma$-absorbing.
\end{theoremi}

Since many assumptions in this theorem are fairly technical at first glance, it may be helpful for the reader to keep in mind some special cases.
For example, the above assumptions on the pair $N\subset G$ are satisfied when the quotient $G/N$ above is isomorphic to either $\IR$ or $\IZ$.
In this case, the theorem states that as long as the action $\alpha$ satisfies a suitable Rokhlin-type criterion relative to $N$, tensorial absorption of the $G$-action $\gamma$ can be detected by restricting to the $N$-actions, even though this restriction precedure (a priori) comes with great loss of dynamical information.
This is most apparent when the normal subgroup $N$ is trivial, which is yet another important special case; see \autoref{cor:dimrokc-cor}:

\begin{cori} \label{Cor-B}
Let $G$ be a second-countable, locally compact group.
Suppose that $G$ contains a discrete, normal, cocompact subgroup that is residually finite and has a box space with finite asymptotic dimension.
Let $A$ be a separable \cstar-algebra with an action $\alpha: G\curvearrowright A$.
Suppose that $\CD$ is a strongly self-absorbing \cstar-algebra with $A\cong A\otimes\CD$.
If the Rokhlin dimension of $\alpha$ with commuting towers is finite, then it follows that $\alpha$ is $\gamma$-absorbing for every semi-strongly self-absorbing action $\gamma: G\curvearrowright\CD$.
\end{cori}

Here it may be useful to keep in mind that any version of what is called the Rokhlin property for $G=\IR$ or $G=\IZ^k$ will automatically imply finite Rokhlin dimension with commuting towers, and is therefore covered by \autoref{Cor-B}.
This is in turn a generalization of \cite[Theorem 1.1]{HirshbergWinter07}, \cite[Theorems 5.8, 5.9]{HirshbergWinterZacharias15}, \cite[Theorem 9.6]{SzaboWuZacharias17}, \cite[Theorem 5.3]{HirshbergSzaboWinterWu17} and \cite[Theorem 4.50(2)]{GardellaLupini16}.
We will in fact only apply the corollary within this paper, with a particular focus on the special case where the action is assumed to have the Rokhlin property.
Some immediate applications of \autoref{Cor-B} will be discussed in Section \ref{sec:applications}.
The main non-trivial application is pursued in Section \ref{sec:multiflows}, which is as follows; see \autoref{thm:multiflows} and \autoref{cor:multiflow-spi}:

\begin{theoremi} \label{Thm-C}
Let $\CD$ be a strongly self-absorbing Kirchberg algebra.
Then up to (very strong) cocycle conjugacy, there is a unique action $\gamma: \IR^k\curvearrowright\CD$ that has finite Rokhlin dimension with commuting towers.
\end{theoremi}

We note that a strongly self-absorbing \cstar-algebra is a Kirchberg algebra precisely when it is traceless.
Kirchberg algebras are (by convention) the separable, simple, nuclear, purely infinite \cstar-algebras, whose celebrated classification is due to Kirchberg--Phillips \cite{KirchbergPhillips00, Phillips00, KirchbergC} and which constitutes a prominent special case of the Elliott classification program.
We note that all other strongly self-absorbing \cstar-algebras are conjectured to be quasidiagonal --- see \cite[Corollary 6.7]{TikuisisWhiteWinter17} --- and so any Rokhlin flows on them would induce Rokhlin flows on the universal UHF algebra, which does not exist; see \cite[page 600]{Kishimoto96_R}, \cite[page 289]{Kishimoto98} or \cite[Section 2]{HirshbergSzaboWinterWu17}.
In particular, the underlying problem above is only interesting to consider in the purely infinite case.

Although the theorem above is not too far off from being a very special case of \cite{Szabo17R} for ordinary flows, this result is entirely new for $k\geq 2$, and is in fact the first classification result for $\IR^k$-actions on \cstar-algebras up to cocycle conjugacy.

The proof goes via induction in the number $k$ of flows generating the action.
In order to achieve a major part of the induction step, the corollary above is used in order to see that any two $\IR^k$-actions as in the statement absorb each other tensorially.
However, in order for this to make sense, it has to be at least established beforehand (as part of the induction step) that any such action has equivariantly approximately inner flip.
This is achieved via a relative Kishimoto-type approximate cohomology-vanishing argument inspired by \cite[Section 3]{Kishimoto02}, which combines arguments related to the Rokhlin property for $\IR^k$-actions with arguments related to the structure theory of semi-strongly self-absorbing actions.

At this moment it seems unclear whether or not to expect a similarly rigid situation for Rokhlin $\IR^k$-actions on general Kirchberg algebras, as is the case for $k=1$ \cite{Szabo17R}.
In general, in order to implement a more general classification of this sort, it would require a technique for both constructing and manipulating cocycles for $\IR^k$-actions (where $k\geq 2$) with the help of the Rokhlin property, which may potentially be much more complicated than for $k=1$.
In essence, our approach based on ideas related to strong self-absorption works because the main result allows one to bypass the need to bother with general cocycles for all of $\IR^k$, but instead requires one only to consider individual copies of $\IR$ inside $\IR^k$ at a time (represented by the flows generating the $\IR^k$-action), enabling an induction process.

In forthcoming work, the full force of the aforementioned main result of this paper (\autoref{thm:dimrokc-main}) will form the basis of further uniqueness results regarding actions of certain discrete amenable groups on strongly self-absorbing \cstar-algebras.
\bigskip

\textbf{Acknowledgements.} The author would like to express his gratitude to the referee for very useful and detailed comments.


\section{Preliminaries}

\begin{nota}
Unless specified otherwise, we will stick to the following notational conventions in this paper:
\begin{itemize}
\item $G$ denotes a loally compact Hausdorff group.
\item $A$ and $B$ denote \cstar-algebras.
\item The symbols $\alpha, \beta, \gamma$ are used to denote point-norm continuous actions on \cstar-algebras.
Since continuity is always assumed in this context, we will simply refer to them as actions. 
\item If $\alpha: G\curvearrowright A$ is an action, then $A^\alpha$ denotes the fixed-point algebra of $A$.
\item If $F$ is a finite subset inside some set $M$, we often denote $F\fin M$.
\item If $(X,d)$ is some metric space with elements $a,b\in X$, then we write $a=_\eps b$ as a shorthand for $d(a,b)\leq\eps$.
\end{itemize}
\end{nota}

We first recall some needed definitions and notation.

\begin{defi}[cf.~{\cite[Definition 3.2]{PackerRaeburn89} and \cite[Section 1]{Szabo18ssa, Szabo17ssa3}}]
Let $\alpha: G\curvearrowright A$ be an action. Consider a strictly continuous map $w: G\to\CU(\CM(A))$.
\begin{enumerate}[label=(\roman*),leftmargin=*] 
\item $w$ is called an $\alpha$-1-cocycle, if one has $w_g\alpha_g(w_h)=w_{gh}$ for all $g,h\in G$.
In this case, the map $\alpha^w: G\to\Aut(A)$ given by $\alpha_g^w=\ad(w_g)\circ\alpha_g$ is again an action, and is called a cocycle perturbation of $\alpha$. 
Two $G$-actions on $A$ are called exterior equivalent if one of them is a cocycle perturbation of the other.
\item Assume that $w$ is an $\alpha$-1-cocycle. 
It is called an approximate coboundary, if there exists a sequence of unitaries $x_n\in\CU(\CM(A))$ such that $x_n\alpha_g(x_n^*) \strict w_g$ for all $g\in G$ and uniformly on compact sets. 
Two $G$-actions on $A$ are called strongly exterior equivalent, if one of them is a cocycle perturbation of the other via an approximate coboundary.
\item Assume $w$ is an $\alpha$-1-cocycle.
It is called an asymptotic coboundary, if there exists a strictly continuous map $x: [0,\infty)\to\CU(\CM(A))$ with $x_0=\eins$ and such that $x_t\alpha_g(x_t^*) \strict w_g$ for all $g\in G$ and uniformly on compact sets. 
Two $G$-actions on $A$ are called very strongly exterior equivalent, if one of them is a cocycle perturbation of the other via an asymptotic coboundary.
\item Let $\beta: G\curvearrowright B$ be another action. 
The actions $\alpha$ and $\beta$ are called cocycle conjugate, written $\alpha\cc\beta$, if there exists an isomorphism $\psi: A\to B$ such that $\psi^{-1}\circ\beta\circ\psi$ and $\alpha$ are exterior equivalent. 
If $\psi$ can be chosen such that $\psi^{-1}\circ\beta\circ\psi$ and $\alpha$ are strongly exterior equivalent, then $\alpha$ and $\beta$ are called strongly cocycle conjugate, written $\alpha\scc\beta$.
If $\psi$ can be chosen such that $\psi^{-1}\circ\beta\circ\psi$ and $\alpha$ are very strongly exterior equivalent, then $\alpha$ and $\beta$ are called very strongly cocycle conjugate, written $\alpha\vscc\beta$.
\end{enumerate}
\end{defi}

Note that for a cocycle $w$, the cocycle identity applied to $g=h=e$ yields $w_e=w_e^2$, and hence $w_e=\eins$.
This is implicitly used in many calculations without further mention.

\begin{defi}[cf.~{\cite[Definition 1.1]{Kirchberg04} and \cite[Section 1]{Szabo18ssa}}] 
Let $A$ be a \cstar-algebra and let $\alpha: G\curvearrowright A$ be an action of a locally compact group. 
\begin{enumerate}[label={\textup{(\roman*)}},leftmargin=*]
\item The sequence algebra of $A$ is given as 
\[
A_\infty = \ell^\infty(\IN,A)/\set{ (x_n)_n ~|~ \lim_{n\to\infty}\| x_n\|=0}.
\]
There is a standard embedding of $A$ into $A_\infty$ by sending an element to its constant sequence. 
We shall always identify $A\subset A_\infty$ this way, unless specified otherwise.
\item Pointwise application of $\alpha$ on representing sequences defines a (not necessarily continuous) $G$-action $\alpha_\infty$ on $A_\infty$. 
Let
\[
A_{\infty,\alpha} = \set{ x\in A_\infty \mid [g\mapsto\alpha_{\infty,g}(x)]~\text{is continuous} }
\]
be the continuous part of $A_\infty$ with respect to $\alpha$.
\item For some \cstar-subalgebra $B\subset A_\infty$, the (corrected) relative central sequence algebra is defined as
\[
F(B,A_\infty) = (A_\infty\cap B')/\ann(B,A_\infty).
\]
\item If $B\subset A_\infty$ is $\alpha_\infty$-invariant, then the $G$-action $\alpha_\infty$ on $A_\infty$ induces a (not necessarily continuous) $G$-action $\tilde{\alpha}_\infty$ on $F(B,A_\infty)$. 
Let
\[
F_\alpha(B,A_\infty) = \set{ y\in F_\alpha(B,A_\infty) \mid [g\mapsto\tilde{\alpha}_{\infty,g}(y)]~\text{is continuous} }
\]
be the continuous part of $F(B,A_\infty)$ with respect to $\alpha$.
\item In case $B=A$, we write $F(A,A_\infty)=F_\infty(A)$ and $F_\alpha(A,A_\infty)=F_{\infty,\alpha}(A)$.
\end{enumerate}
\end{defi}

\begin{defi}[\see{BarlakSzabo16ss}{Definition 3.3}]
Let $G$ be a second-countable, locally compact group, and let $\alpha: G\curvearrowright A$ and $\beta: G\curvearrowright B$ be actions on separable \cstar-algebras.
An equivariant $*$-homomorphism $\phi: (A,\alpha)\to (B,\beta)$ is called (equivariantly) sequentially split, if there exists a $*$-homomorphism $\psi: (B,\beta)\to (A_{\infty,\alpha},\alpha_\infty)$ such that $\psi(\phi(a))=a$ for all $a\in A$.
\end{defi}

\begin{defi}
Let $G$ be a second-countable, locally compact group, and let $\alpha: G\curvearrowright A$ and $\beta: G\curvearrowright B$ be actions on unital \cstar-algebras.
Let $\phi_1, \phi_2: (A,\alpha)\to (B,\beta)$ be two unital and equivariant $*$-homomorphisms.
We say that $\phi_1$ and $\phi_2$ are approximately $G$-unitarily equivalent, if the following holds.
For every $\CF\fin A$, $\eps\greater 0$, and compact set $K\subseteq G$, there exists a unitary $v\in\CU(B)$ such that
\[
\max_{a\in\CF}~\norm \phi_2(a)-v\phi_1(a)v^* \norm \leq\eps,\quad \max_{g\in K}~\norm v-\beta_g(v) \norm \leq \eps.
\]
\end{defi}

\begin{defi}[\see{Szabo18ssa}{Definitions 3.1, 4.1}]
Let $\CD$ be a separable, unital \cstar-algebra and $G$ a second-countable, locally compact group. Let $\gamma: G\curvearrowright\CD$ be an action. 
We say that
\begin{enumerate}[label={(\roman*)},leftmargin=*]
\item $\gamma$ is a strongly self-absorbing action, if the equivariant first-factor embedding
\[
\id_\CD\otimes\eins_\CD: (\CD,\gamma)\to (\CD\otimes\CD,\gamma\otimes\gamma)
\]
is approximately $G$-unitarily equivalent to an isomorphism.
\item $\gamma$ is semi-strongly self-absorbing, if it is strongly cocycle conjugate to a strongly self-absorbing action.
\end{enumerate}
\end{defi}

\begin{defi}[see {\cite[Definition 2.17]{Szabo18ssa2}}]
Let $G$ be a second-countable, locally compact group.
An action $\alpha: G\curvearrowright A$ on a unital \cstar-algebra is called unitarily regular, if for every $\eps\greater 0$ and compact set $K\subseteq G$, there exists $\delta\greater 0$ such that for every pair of unitaries
\[
u,v\in\CU(A) \quad\text{with}\quad \max_{g\in K}~ \max\set{ \norm \alpha_g(u)-u \norm, \norm \alpha_g(v)-v \norm} \leq\delta,
\]
there exists a continuous path of unitaries $w: [0,1]\to\CU(A)$ satisfying
\[
w(0)=\eins,\quad w(1)=uvu^*v^*,\quad \max_{0\leq t\leq 1}~\max_{g\in K}~ \norm \alpha_g(w(t))-w(t) \norm \leq\eps. 
\]
\end{defi}

Let us recall some of the main results from \cite{Szabo18ssa, Szabo18ssa2, Szabo17ssa3}, which we will use throughout.
We will also use the perspective given in \cite[Section 4]{BarlakSzabo16ss}.

\begin{theorem}[cf.~{\cite[Theorems 3.7, 4.7]{Szabo18ssa}}]
\label{thm:equi-McDuff}
Let $G$ be a second-countable, locally compact group. Let $A$ be a separable \cstar-algebra and $\alpha: G\curvearrowright A$ an action. Let $\CD$ be a separable, unital \cstar-algebra and $\gamma: G\curvearrowright\CD$ a semi-strongly self-absorbing action. 
The following are equivalent:
\begin{enumerate}[label=\textup{(\roman*)},leftmargin=*] 
\item $\alpha$ and $\alpha\otimes\gamma$ are strongly cocycle conjugate. \label{equi-McDuff1}
\item $\alpha$ and $\alpha\otimes\gamma$ are cocycle conjugate.  \label{equi-McDuff2}
\item There exists a unital, equivariant $*$-homomorphism from $(\CD,\gamma)$ to $\big( F_{\infty,\alpha}(A), \tilde{\alpha}_\infty \big)$. \label{equi-McDuff3}
\item The equivariant first-factor embedding $\id_A\otimes\eins: (A,\alpha)\to (A\otimes\CD,\alpha\otimes\gamma)$ is sequentially split. \label{equi-McDuff4}
\end{enumerate}
If $\gamma$ is moreover unitarily regular, then these statements are equivalent to
\begin{enumerate}[label=\textup{(\roman*)},leftmargin=*,resume]
\item $\alpha$ and $\alpha\otimes\gamma$ are very strongly cocycle conjugate. \label{equi-McDuff5}
\end{enumerate}
\end{theorem}

\begin{reme}
For the rest of this paper, an action $\alpha$ satisfying condition \ref{equi-McDuff1} from above is called $\gamma$-absorbing or $\gamma$-stable. In the particular case that $\gamma$ is the trivial $G$-action on a strongly self-absorbing \cstar-algebra $\CD$, we will say that $\alpha$ is equivariantly $\CD$-stable.
\end{reme}

\begin{rem} \label{rem:unitary-reg}
Unitary regularity for an action is a fairly mild technical assumption.
It can be seen as the equivariant analog of the \cstar-algebraic property that the commutator subgroup inside the unitary group lies in the connected component of the unit.
Unitary regularity holds automatically under equivariant $\CZ$-stability, but also in other cases; see \cite[Proposition 2.19 and Example 6.4]{Szabo18ssa2}.
\end{rem}

\begin{theorem}[see {\cite[Theorem 5.9]{Szabo18ssa2}}] \label{thm:ssa-ext}
A semi-strongly self-absorbing action $\gamma: G\curvearrowright\CD$ is unitarily regular if and only if the class of all separable $\gamma$-absorbing $G$-\cstar-dynamical systems is closed under equivariant extensions.
\end{theorem}

We will extensively use the following without much mention:

\begin{prop}[see {\cite{Brown00}}]
\label{prop:cont-ell}
Let $G$ be a second-countable, locally compact group. Let $A$ be a \cstar-algebra and $\alpha: G\curvearrowright A$ an action. Let $x\in A_{\infty,\alpha}$ and $(x_n)_n\in\ell^\infty(\IN,A)$ a bounded sequence representing $x$. Then $(x_n)_n$ is a continuous element with respect to the componentwise action of $\alpha$ on $\ell^\infty(\IN,A)$.
\end{prop}


\section{Box spaces and partitions of unity over groups}
\label{sec:box}

\begin{defi} \label{def:res-cp-approx}
Let $G$ be a second-countable, locally compact group. A residually compact approximation of $G$ is a decreasing sequence $H_{n+1}\subseteq H_n\subseteq G$ of normal, discrete, cocompact subgroups in $G$ with $\bigcap_{n\in\IN} H_n=\set{1}$.
If $G$ is a discrete group, then the subgroups $H_n$ will have finite index, in which case we call the sequence $(H_n)_n$ a residually finite approximation.
\end{defi}

\begin{rem}
In the above setting, the sequence $(H_n)_n$ is automatically a residually finite approximation of the discrete group $H_1$.
\end{rem}

Recall the definition of a box space; see \cite[Definition 10.24]{RoeCG} or \cite{Khukhro12}.

\begin{defi}
Let $\Gamma$ be a countable discrete group and $\CS=(H_n)_n$ a residually finite approximation of $\Gamma$. Let $d$ be a proper, right-invariant metric on $\Gamma$. For every $n\in\IN$, denote by $\pi_n: \Gamma\to \Gamma/H_n$ the quotient map, and $\pi_{n*}(d)$ the push-forward metric on $\Gamma/H_n$ that is induced by $d$. The box space of $\Gamma$ along $\CS$, denoted $\square_\CS \Gamma$, is the coarse disjoint union of the sequence of finite metric spaces $\big( \Gamma/H_n, \pi_{n*}(d) \big)$.
\end{defi}

The main purpose of this section will be to prove the following technical lemma:

\begin{lemma} \label{lem:decay}
Let $G$ be a second-countable, locally compact group and $\CS=(H_n)_n$ a residually compact approximation of $G$. 
Asssume that the box space $\square_\CS H_1$ has finite asymptotic dimension $d$. 
Then for every $\eps>0$ and compact set $K\subset G$, there exists $n\in\IN$ and continuous, compactly supported functions $\mu^{(0)},\dots,\mu^{(d)}: G\to [0,1]$ satisfying:
\begin{enumerate}[label={\textup{(\alph*)}},leftmargin=*]
\item for every $l=0,\dots,d$ and $h\in H_n\setminus\set{1}$, we have
\[
\supp(\mu^{(l)})\cap \supp(\mu^{(l)})\cdot h = \emptyset;
\] \label{lem:decay:a}
\item for every $g\in G$, we have
\[
\sum_{l=0}^d \sum_{h\in H_n} \mu^{(l)}(gh) = 1;
\] \label{lem:decay:b}
\item for every $l=0,\dots,d$ and $g\in K$, we have
\[
\|\mu^{(l)}(g\cdot\_)-\mu^{(l)}\|_\infty \leq \eps.
\] \label{lem:decay:c}
\end{enumerate}
\end{lemma}

\begin{rem}
In the case that $G=\Gamma$ is a discrete group and $\CS$ is a residually finite approximation, this is precisely \cite[Lemma 2.13]{SzaboWuZacharias17}.
In order to prove \autoref{lem:decay}, we shall convince ourselves that the desired functions can be constructed from finitely supported functions with similar properties on the cocompact subgroup $H_1$.
For this, we first have to observe a slightly improved version of \cite[Lemma 2.13]{SzaboWuZacharias17} in the discrete case.
\end{rem}

\begin{lemma} \label{lem:decay-discrete}
Let $\Gamma$ be a countable discrete group and $\CS=(H_n)_n$ a residually finite approximation of $\Gamma$. 
Asssume that the box space $\square_\CS \Gamma$ has finite asymptotic dimension $d$. 
Then for every $\eps>0$ and finite set $F\fin \Gamma$, there exists $n\in\IN$ and finitely supported functions $\nu^{(0)},\dots,\nu^{(d)}: \Gamma\to [0,1]$ satisfying:
\begin{enumerate}[label={\textup{(\alph*)}},leftmargin=*]
\item for every $l=0,\dots,d$ and $h\in H_n\setminus\set{1}$, we have
\[
g_1hg_2^{-1} \notin F \quad\text{for all } g_1,g_2\in\supp(\nu^{(l)});
\] \label{lem:decay-discrete:a}
\item for every $g\in \Gamma$, we have
\[
\sum_{l=0}^d \sum_{h\in H_n} \nu^{(l)}(gh) = 1;
\] \label{lem:decay-discrete:b}
\item for every $l=0,\dots,d$ and $g\in F$, we have
\[
\|\nu^{(l)}(g\cdot\_)-\nu^{(l)}\|_\infty \leq \eps.
\] \label{lem:decay-discrete:c}
\end{enumerate}
\end{lemma}
\begin{proof}
Let $\eps\greater 0$ and $F\subset G$ be given.
We apply \cite[Lemma 2.13]{SzaboWuZacharias17} and choose some $n$ and finitely supported functions $\theta^{(0)},\dots,\theta^{(d)}: \Gamma\to [0,1]$ satisfying
\begin{equation} \label{eq:decay-discrete:1}
\supp(\theta^{(l)})\cap\supp(\theta^{(l)})\cdot h_n=\emptyset\quad\text{for all } h_n\in H_n\setminus\set{1};
\end{equation} 
as well as properties \ref{lem:decay-discrete:b} and \ref{lem:decay-discrete:c}.
Combining property \eqref{eq:decay-discrete:1} and \ref{lem:decay-discrete:c}, we see that if $g_1, g_2\in\supp(\theta^{(l)})$ and $h\in H_n\setminus\set{1}$ are such that $g_1hg_2^{-1}=g_1(g_2h^{-1})\in F$, then we get
\begin{equation} \label{eq:decay-discrete:2}
\vslash \theta^{(l)}(g_1) \vslash = \vslash \theta^{(l)}(g_1 hg_2^{-1} \cdot g_2 h^{-1}) \vslash \stackrel{\ref{lem:decay-discrete:c}}{\leq} \eps+\vslash \theta^{(l)}(g_2h^{-1}) \vslash \stackrel{\eqref{eq:decay-discrete:1}}{=} \eps. 
\end{equation}
Let us define new functions $\kappa^{(l)}: \Gamma\to [0,1]$ via
\begin{equation} \label{eq:decay-discrete:3}
\kappa^{(l)}(g) = \big( \theta^{(l)}(g)-\eps \big)_+ .
\end{equation}
These new functions clearly still satisfy property \ref{lem:decay-discrete:c}.
For any $g_1, g_2\in \supp(\kappa^{(l)})$, we evidently have $g_1, g_2\in\supp(\theta^{(l)})$, so assuming $g_1hg_2^{-1}\in F$ for some $h\in H_n\setminus\set{1}$ would imply $\kappa^{(l)}(g_1)=0$ by \eqref{eq:decay-discrete:2} and \eqref{eq:decay-discrete:3}, a contradiction.
In particular we obtain property \ref{lem:decay-discrete:a} for these functions.

Lastly, note that property \ref{lem:decay-discrete:a} implies that any sum as in \ref{lem:decay-discrete:b} can have at most $d+1$ non-vanishing summands, and thus we may estimate for all $g\in\Gamma$ that
\[
\begin{array}{ccl}
1 &=& \dst\sum_{l=0}^d \sum_{h\in H_n} \theta^{(l)}(gh) \\
&\geq& \dst\sum_{l=0}^d \sum_{h\in H_n} \kappa^{(l)}(gh) \\
&\geq& \dst\Big( \sum_{l=0}^d \sum_{h\in H_n} \theta^{(l)}(gh) \Big)-(d+1)\eps \\
&=& 1-(d+1)\eps.
\end{array}
\]
So let us yet again define new functions $\nu^{(l)}: \Gamma\to [0,1]$ via
\[
\nu^{(l)}(g) = \Big( \sum_{l=0}^d \sum_{h\in H_n} \kappa^{(l)}(gh) \Big)^{-1} \kappa^{(l)}(g).
\]
By our previous calculation, we have $\kappa^{(l)} \leq \nu^{(l)} \leq \frac{1}{1-(d+1)\eps}\kappa^{(l)}$.
For these functions, property \ref{lem:decay-discrete:a} will still hold, while property \ref{lem:decay-discrete:b} holds by construction.
Moreover property \ref{lem:decay-discrete:c} holds with regard to the tolerance
\[
\eta_\eps := \eps + \frac{2(d+1)\eps}{1-(d+1)\eps}
\]
in place of $\eps$.
Since $\eta_\eps\to 0$ as $\eps\to 0$, this means that the functions $\nu^{(l)}$ will have the desired property after rescaling $\eps$.
This shows our claim.
\end{proof}

\begin{lemma} \label{lem:cutoff-function}
Let $G$ be a locally compact group and $H\subset G$ a closed and cocompact subgroup. Let $\mu$ be a left-invariant Haar measure on $H$. 
\begin{enumerate}[label={\textup{(\roman*)}},leftmargin=*]
\item There exists a compactly supported continuous function $C: G\to [0,\infty)$ satisfying the equation
\[
\int_H C(gh)~d\mu(h)=1\quad\text{for all}~g\in G.
\] \label{lem:cutoff-function:1}
\item Assume furthermore that $G$ is amenable. Let $\eps\greater 0$ and let $K\subset G$ be a compact subset.
Then there exists a function $C$ as above with the additional property that
\[
\norm C(g\cdot\_) - C \norm_\infty \leq \eps. 
\] \label{lem:cutoff-function:2}
\end{enumerate}
\end{lemma}
\begin{proof}
\ref{lem:cutoff-function:1}: As $H$ is a cocompact subgroup, there exists some compact set $K_{H}\subset G$ such that $G=K_H\cdot H$.
By Urysohn-Tietze, we may choose a compactly supported continuous function $c: G\to [0,1]$ with $c|_{K_{H}}=1$. 
Define the compact set $K_c\subset H$ via
\[
K_c = \Big( K_{H}^{-1}\cdot\supp(c) \Big)\cap H.
\]
Then for every $g\in G$, there is some $h_0\in H$ with $gh_0\in K_{H}$.
We have 
\[
\begin{array}{ccl}
\supp\big(c(gh_0\cdot\_)\big)\cap H &=&  \Big( (gh_0)^{-1}\cdot\supp(c) \Big)\cap H \\
&\subset& K_c.
\end{array}
\]
Thus, we get that
\[
0 < \int_H c(gh)~d\mu(h) = \int_H c(gh_0h)~d\mu(h) \leq \mu(K_c)<\infty.
\]
Note that by the properties of the Haar measure, the assignment
\[
\CI : G\to (0,\infty),\quad g\mapsto \int_H c(gh)~d\mu(h)
\]
is $H$-periodic. 
Then the above computation shows that this assignment yields a well-defined, continuous function on $G$, which by $H$-periodicity and cocompactness of $H$ can be viewed as a continuous function on the compact space $G/H$.
Thus the image of this function is compact.
In particular, its (pointwise) multiplicative inverse is also bounded and continuous. Let us define
\[
C: G\to [0,\infty),\quad g\mapsto \CI(g)^{-1} c(g).
\]
Then this again yields a continuous function on $G$ with compact support, but with the property that
\begin{equation} \label{eq:cutoff:1}
\int_H C(gh)~d\mu(h)=1\quad\text{for all}~g\in G.
\end{equation}

\ref{lem:cutoff-function:2}: Let us now additionally assume that $G$ is amenable.
Let $\eps\greater 0$ and $K\subset G$ be given as in the statement.
Let $\rho^G$ denote a right-invariant Haar measure on $G$.
It follows from \cite{EmersonGreenleaf67} that we may find some compact set $J\subset G$ with $\rho^G(J)\greater 0$ such that $\rho^G\big( J\Delta (J\cdot K) \big)\leq \frac{\eps}{\norm C\norm_\infty}\cdot \rho^G(J)$.
Define $C': G\to [0,\infty)$ via
\[
C'(g)=\frac{1}{\rho^G(J)} \cdot \int_J C(xg) ~d\rho^G(x).
\]
Clearly $C'$ is yet another continuous function with compact support contained in $J^{-1}\cdot\supp(C)$.
Given any element $g\in G$, we compute
\[
\begin{array}{ccl}
\dst \int_H C'(gh)~d\mu(h) &=& \dst \int_H \frac{1}{\rho^G(J)} \Big( \int_J C(xgh)~d\rho^G(x) \Big) ~d\mu(h) \\
&=& \dst\frac{1}{\rho^G(J)} \int_J \Big( \int_H C(xgh)~d\mu(h) \Big)~d\rho^G(x) \\
&\stackrel{\eqref{eq:cutoff:1}}{=}& \dst \frac{1}{\rho^G(J)} \int_J 1 ~d\rho^G(x) \ = \ 1.
\end{array}
\]
Furthermore, we have for any $g_K\in K$ and $g\in G$ that
\[
\begin{array}{ccl}
\vslash C'(g_Kg)-C'(g) \vslash &=& \dst \frac{1}{\rho^G(J)} \cdot \Big\vslash \int_J C(xg_Kg) ~d\rho^G(x) - \int_J C(xg) ~d\rho^G(x)\Big\vslash \\
&\leq& \dst \frac{1}{\rho^G(J)} \cdot \norm C\norm_\infty\cdot \rho^G(J\Delta Jg_K) \\
&\leq& \eps.
\end{array}
\]
This shows the last part of the claim.
\end{proof}

\begin{proof}[Proof of \autoref{lem:decay}]
We first remark that since the box space $\square_\CS H_1$ has finite asymptotic dimension, it also has property A, and therefore $H_1$ is amenable; see \cite[Theorems 4.3.6 and 4.4.6]{NowakYuLSG} and \cite[Proposition 11.39]{RoeCG}.
As $H_1$ is a discrete cocompact normal subgroup in $G$, we also see that $G$ is amenable.

Let $\eps\greater 0$ and $K\subset G$ be given.
Then there exists a function $C: G\to [0,\infty)$ as in \autoref{lem:cutoff-function} for $H_1$ in place of $H$, with the property that
\begin{equation} \label{eq:decay:1}
\norm C(g\cdot\_)-C\norm_\infty \leq \eps \quad\text{for all } g\in K.
\end{equation}
Let us denote the support of $C$ by $S=\supp(C)$.
As $H_1$ is discrete in $G$ and $S$ is compact, there exists a finite set $F\subset H_1$ with 
\begin{equation} \label{eq:decay:2}
h_1\in F \quad\text{whenever}\quad h_1\in H \text{ and } S\cap Sh_1\neq\emptyset.
\end{equation}
Applying \autoref{lem:decay-discrete}, there exists some $n$ and finitely supported functions $\nu^{(0)},\dots,\nu^{(d)}: H_1\to [0,1]$ satisfying the following properties:\footnote{Note that we will reserve the notation $h_1, h_2$ for elements in $H_1$, whereas $h_n$ will denote an element in the smaller subgroup $H_n$ for $n\greater 2$.}
\begin{equation} \label{eq:decay:3}
h_1h_nh_2^{-1}\notin F \quad\text{for all } h_1,h_2\in\supp(\nu^{(l)}) \text{ and } h_n\in H_n\setminus\set{1};
\end{equation}
\begin{equation} \label{eq:decay:4}
1=\sum_{l=0}^d \sum_{h_n\in H_n} \nu^{(l)}(h_1h_n) \quad\text{for all } h_1\in H_1.
\end{equation}
We define $\mu^{(l)}: G\to [0,\infty)$ for $l=0,\dots,d$ via
\[
\mu^{(l)}(g) = \sum_{h_1\in H_1} C(gh_1^{-1})\nu^{(l)}(h_1).
\]
Since $\nu^{(l)}$ is finitely supported on $H_1$, we see that $\mu^{(l)}$ is a finite sum of continuous functions with compact support, and hence $\mu^{(l)}\in\CC_c(G)$.

We claim that these functions have the desired properties.
Let us verify \ref{lem:decay:a}, which is equivalent to the statement that
\[
\mu^{(l)}(g)\cdot \mu^{(l)}(gh_n^{-1})=0\quad\text{for all }g\in G \text{ and } h_n\in H_n\setminus\set{1}.
\]
Fix an element $h_n\in H_n\setminus\set{1}$ for the moment.
We compute
\[
\begin{array}{ccl}
\mu^{(l)}(g)\cdot \mu^{(l)}(gh_n^{-1}) &=& \dst\sum_{h_1, h_2\in H_1} C(gh_1^{-1}) C(gh_n^{-1}h_2^{-1}) \nu^{(l)}(h_1) \nu^{(l)}(h_2) \\
&=& \dst\sum_{h_1, h_2\in H_1} C(gh_1^{-1}) C(gh_2^{-1}) \nu^{(l)}(h_1) \nu^{(l)}(h_2h_n^{-1})
\end{array}
\]
We claim that each individual summand is zero.
Indeed, suppose $h_1, h_2\in H_1$ are such that $\nu^{(l)}(h_1) \nu^{(l)}(h_2 h_n^{-1}) \greater 0$.
Then $h_1\in\supp(\nu^{(l)})$ and $h_2\in\supp(\nu^{(l)})\cdot h_n$, which implies $h_1h_2^{-1}\notin F$ by \eqref{eq:decay:3}.
By our choice of $F$, we obtain
\[
\begin{array}{ccl}
\supp(C(\_\cdot h_1^{-1}))\cap\supp(C(\_\cdot h_2^{-1})) &\subseteq& Sh_1\cap Sh_2 \\
&=& (Sh_1h_2^{-1}\cap S)\cdot h_2 \\
&\stackrel{\eqref{eq:decay:2}}{=}& \emptyset,
\end{array}
\]
and in particular $C(gh_1^{-1})C(gh_2^{-1})=0$.
This finishes the proof that each summand of the above sum is zero and shows property \ref{lem:decay:a}.

Let us now show property \ref{lem:decay:b}.
We calculate for every $g\in G$ that
\[
\begin{array}{ccl}
\dst\sum_{l=0}^d \sum_{h_n\in H_n} \mu^{(l)}(gh_n) &=& \dst\sum_{l=0}^d \sum_{h_n\in H_n} \sum_{h_1\in H_1} C(gh_nh_1^{-1})\nu^{(l)}(h_1) \\
&=& \dst\sum_{l=0}^d \sum_{h_n\in H_n} \sum_{h_1\in H_1} C(gh_1^{-1})\nu^{(l)}(h_1h_n) \\
&=& \dst\sum_{h_1\in H_1} C(gh_1^{-1}) \Big( \sum_{l=0}^d \sum_{h_n\in H_n} \nu^{(l)}(h_1h_n) \Big) \\
&\stackrel{\eqref{eq:decay:4}}{=}& \dst\sum_{h_1\in H_1} C(gh_1) \\
&\stackrel{\ref{lem:cutoff-function}}{=}& 1 .
\end{array}
\]
Let us now turn to \ref{lem:decay:c}.
Given any $g\in G$ and $g_K\in K$, we compute
\[
\begin{array}{ccl}
\vslash \mu^{(l)}(g_Kg)-\mu^{(l)}(g)\vslash &=& \dst\Big\vslash \sum_{h_1\in H_1} \big( C(g_Kgh_1^{-1})-C(gh_1^{-1}) \big) \nu^{(l)}(h_1) \Big\vslash \\
&\stackrel{\eqref{eq:decay:4}}{\leq}& \dst \sup_{h_1\in H_1} \vslash C(g_Kgh_1^{-1})-C(gh_1^{-1}) \vslash \\
&\leq& \dst\norm C(g_K\cdot\_)-C\norm_\infty \ \stackrel{\eqref{eq:decay:2}}{\leq} \ \eps.
\end{array}
\]
As $g\in G$ was arbitrary, this finishes the proof.
\end{proof}

\begin{reme}
Let $G$ be a locally compact group and $H\subset G$ a closed, cocompact subgroup.
For any \cstar-algebra $A$, we may naturally view $\CC(G/H,A)$ as a \cstar-subalgebra of (right-)$H$-periodic functions inside $\CC_b(G,A)$ by assigning a function $f$ to the function $f'$ given by $f'(g)=f(gH)$.

In what follows, we will briefly establish a technical result that allows one to perturb {\it approximately} $H$-periodic functions in $\CC_b(G,A)$ to {\it exactly} $H$-periodic functions in a systematic way.
\end{reme}

\begin{lemma} \label{lem:periodic-expect}
Let $G$ be a locally compact group and $H\subset G$ a closed, cocompact subgroup. 
Let $A$ be a \cstar-algebra. 
Then there exists a conditional expectation $E: \CC_b(G, A)\to\CC(G/H, A)$ with the following property.

For every $\eps>0$ and compact set $K\subset G$, there exists $\delta>0$ and a compact set $J\subset H$ such that the following holds: 

If $f\in\CC_b(G,A)$ satisfies
\[
\max_{g\in K}~\max_{h\in J}~\|f(g)-f(gh)\|\leq\delta,
\]
then
\[
\|f-E(f)\|_{\infty,K}\leq\eps.
\]
\end{lemma}
\begin{proof}
Let $\mu$ be a left-invariant Haar measure on $H$. Let $C\in\CC_c(G)$ be a function as in \autoref{lem:cutoff-function}. Then we define
\[
E: \CC_b(G,A)\to\CC(G/H,A),\quad E(f)(gH)=\int_H C(gh)f(gh)~d\mu(h).
\]
Since $C$ is compactly supported and the Haar measure $\mu$ is left-invariant, it is clear that $E$ is well-defined and indeed a conditional expectation.
Let $\eps>0$ and $K\subset G$ be given. 
Let $S$ be the compact support of $C$. 
Then the set $J:= (K^{-1}S)\cap H$ is compact in $H$ with the property that
\begin{equation} \label{eq:periodic-expect:1}
g\in K ~\text{and}~ gh\in S \implies h\in J
\end{equation}
for all $h\in H$. Set 
\[
\delta=\frac{\eps}{1+\mu(J)\cdot\|C\|_\infty}.
\] 
For every $f\in\CC_b(G,A)$ with
\[
\max_{g\in K}~\max_{h\in J}~\|f(g)-f(gh)\|\leq\delta,
\]
it follows for every $g\in K$ that
\[
\begin{array}{cl}
\multicolumn{2}{l}{ \|f(g)-E(f)(gH)\| } \\
=& \dst \Big\| \Big(\int_H C(gh)~d\mu(h) \Big) f(g)-\int_H C(gh)f(gh)~d\mu(h) \Big\| \\
\stackrel{\eqref{eq:periodic-expect:1}}{=}& \dst \Big\| \int_J C(gh)(f(g)-f(gh)) ~d\mu(h) \Big\| \\
\leq& \dst \mu(J)\cdot\|C\|_\infty\cdot\delta ~\leq~ \eps.
\end{array}
\]
This shows our claim.
\end{proof}

\begin{cor} \label{cor:cpc-perturb}
Let $G$ be a locally compact group and $H\subset G$ a closed, cocompact subgroup. Let $A$ and $B$ be two \cstar-algebras. 
Then for every $\eps>0$, $F\fin B$ and compact set $K\subset G$, there exists $\delta>0$ and a compact set $J\subset H$ such that the following holds: 

If $\Theta: B\to\CC_b(G,A)$ is a c.p.c.\ map with
\[
\max_{g\in K}~\max_{h\in J}~\|\Theta(b)(g)-\Theta(b)(gh)\|\leq\delta \quad\text{for all}~b\in F,
\]
then there exists a c.p.c.\ map $\Psi: B\to\CC(G/H,A)$ with
\[
\max_{g\in K}~\|\Psi(b)(gH)-\Theta(b)(g)\|\leq\eps \quad\text{for all}~b\in F.
\]
\end{cor}
\begin{proof}
Let $E: \CC_b(G,A)\to\CC(G/H,A)$ be a conditional expectation as in \autoref{lem:periodic-expect}.
Given a triple $(\eps,F,K)$, choose $\delta>0$ and $J\subset H$ so that the property in \autoref{lem:periodic-expect} holds for all $f\in\CC_b(G,A)$ with respect to the pair $(\eps,K)$.
Then we can directly conclude that if $\Theta$ is a map as in the statement, then $\Psi=E\circ\Theta$ has desired property.
\end{proof}


\section{Systems generated by order zero maps with commuting ranges}

The following notation and observations are \cite[Lemma 6.6]{HirshbergSzaboWinterWu17} and originate in \cite[Section 5]{HirshbergWinterZacharias15}.

\begin{nota} \label{nota:multi-cones}
Let $D_1,\dots, D_n$ be finitely many unital \cstar-algebras.
For $t\in [0,1]$ and $j=1,\dots,n$, we denote
\[
D_j^{(t)} := \begin{cases} D_j &,\quad t>0, \\ \IC\cdot\eins_{D_j} &,\quad t=0. \end{cases}
\]
Given moreover a tuple $\vec{t}=(t_1,\dots,t_n)\in [0,1]^n$, let us denote 
\[
D^{(\vec{t})} := D_1^{(t_1)}\otimes_{\max} D_2^{(t_2)}\otimes_{\max}\dots\otimes_{\max} D_n^{(t_n)}.
\]
Consider the simplex
\[
\Delta^{(n)} := \set{ \vec{t}\in [0,1]^n \mid t_1+\dots+t_n=1}
\]
and set
\[
\CE(D_1,\dots,D_n) := \set{ f\in \CC\Big( \Delta^{(n)}, D_1\otimes_{\max}\dots\otimes_{\max} D_n  \Big) \ \Big| \ f(\vec{t})\in D^{(\vec{t})} }.
\]
In the case that $D_j=D$ are all the same \cstar-algebra, we will write 
\[
\CE(D_1,\dots,D_n) =: \CE(D,n)
\] 
instead.
For every $j=1,\dots,n$, we will consider the canonical c.p.c.\ order zero map
\[
\eta_j: D_j \to \CE(D_1,\dots,D_n)
\]
given by 
\[
\eta_j(d_j)(\vec{t})=t_j\cdot (\eins_{D_1}\otimes\dots\otimes\eins_{D_{j-1}}\otimes d_j\otimes\eins_{D_{j+1}}\otimes\dots\otimes\eins_{D_n}).
\]
One easily checks that the ranges of the maps $\eta_j$ generate $\CE(D_1,\dots,D_n)$ as a \cstar-algebra.
\end{nota}

\begin{prop} \label{prop:multi-cones-property}
Let $D_1,\dots, D_n$ be unital \cstar-algebras.
Then the \cstar-algebra $\CE(D_1,\dots,D_n)$ together with the c.p.c.\ order zero maps $\eta_j: D_j\to\CE(D_1,\dots,D_n)$ satisfies the following universal property:

If $B$ is any unital \cstar-algebra and $\psi_j: D_j\to B$ for $j=1,\dots,n$ are c.p.c.\ order zero maps with pairwise commuting ranges and 
\[
\psi_1(\eins_{D_1})+\dots+\psi_n(\eins_{D_n})=\eins_B,
\]
then there exists a unique unital $*$-homomorphism $\Psi: \CE(D_1,\dots,D_n)\to B$ such that $\Psi\circ\eta_j = \psi_j$ for all $j=1,\dots,n$.
\end{prop}

\begin{nota} \label{nota:multi-cones-actions}
Let $G$ be a second-countable, locally compact group.
Let $D_1,\dots,D_n$ be unital \cstar-algebras with continuous actions $\alpha^{(j)}: G\curvearrowright D_j$ for $j=1,\dots,n$.
Then the $G$-action on $\CC\Big( \Delta^{(n)}, D_1\otimes_{\max}\dots\otimes_{\max} D_n \Big)$ defined fibrewise by $\alpha^{(1)}\otimes_{\max}\dots\otimes_{\max}\alpha^{(n)}$ restricts to a well-defined action
\[
\CE(\alpha^{(1)},\dots,\alpha^{(n)}): G\curvearrowright \CE(D_1,\dots,D_n)
\]
We will again denote $\CE(\alpha,n):=\CE(\alpha^{(1)},\dots,\alpha^{(n)})$ in the special case that all $(D_j,\alpha^{(j)})=(D,\alpha)$ are the same \cstar-dynamical system.
\end{nota}

\begin{rem} \label{rem:multi-cones-equi-property}
By the universal property in \autoref{prop:multi-cones-property}, the $G$-action $\CE(\alpha^{(1)},\dots,\alpha^{(n)})$ defined in \autoref{nota:multi-cones-actions} is uniquely determined by the identity $\CE(\alpha^{(1)},\dots,\alpha^{(n)})_g\circ\eta_j = \eta_j\circ\alpha^{(j)}_g$ for all $j=1,\dots,n$ and $g\in G$.

This immediately allows us obtain the following equivariant version of \autoref{prop:multi-cones-property} as a consequence:

Let $B$ be any unital \cstar-algebra with an action $\beta: G\curvearrowright B$.
If $\psi_j: (D_j,\alpha^{(j)})\to (B,\beta)$ are equivariant c.p.c.\ order zero maps with pairwise commuting ranges and $\psi_1(\eins_{D_1})+\dots+\psi_n(\eins_{D_n})=\eins_B$, then there exists a unique unital equivariant $*$-homomorphism
\[
\Psi: \Big( \CE(D_1,\dots,D_n) ,\CE(\alpha^{(1)},\dots,\alpha^{(n)}) \Big) \to (B,\beta)
\]
satisfying $\Psi\circ\eta_j = \psi_j$ for all $j=1,\dots,n$.
\end{rem}

\begin{rem} \label{rem:multi-cones-join}
Let us now also convince ourselves of a different natural way to view the \cstar-algebras from \autoref{nota:multi-cones}.

For this, let us first consider the case $n=2$, so we have two unital \cstar-algebras $D_1$ and $D_2$.
Notice that $[0,1]$ is naturally homeomorphic to the simplex $\Delta^{(2)}=\set{ (t_1,t_2) \in [0,1]^{2} \mid t_1+t_2=1 }$ via the assignment $t\mapsto (t,t-1)$.
In this way we may see that there is a natural isomorphism
\[
\begin{array}{cl}
\multicolumn{2}{l}{ \CE(D_1,D_2) }\\
\stackrel{\mathrm{def}}{=} & \set{ f\in\CC(\Delta^{(2)}, D_1\otimes_{\max} D_2) \mid f(0,1) \in D_1\otimes\eins,\ f(1,0)\in\eins\otimes D_2 } \\
\cong& \set{ f\in \CC([0,1],D_1\otimes_{\max} D_2) \mid f(0)\in D_1\otimes\eins,\ f(1)\in\eins\otimes D_2 } \\
=: & D_1\star D_2.
\end{array}
\]
In particular, we see that the notation $\CE(D_1,D_2)$ is consistent with \cite[Definition 5.1]{Szabo18ssa2}. 
As pointed out in \cite[Remark 5.2]{Szabo18ssa2}, the assignment $(D_1, D_2)\mapsto \CE(D_1,D_2)$ on pairs of unital \cstar-algebras therefore generalizes the join construction for pairs of compact spaces, which gives rise to the notation $D_1\star D_2$.

Let now $n\geq 2$ and let $D_1,\dots, D_{n+1}$ be unital \cstar-algebras.
The simplex $\Delta^{(n+1)}$ is homeomorphic to $[0,1]\times\Delta^{(n)}$ via the assignment
\[
(t_1,\vec{t}) \mapsto \begin{cases} (1,\vec{t}) &,\quad t_{1}=0 \\ \big( 1-t_1, \frac{\vec{t}}{1-t_1} \big) &, \quad t_1\neq 0 \end{cases}
\]
for $(\vec{t},t_{n+1})\in\Delta^{(n+1)}$.
Keeping this in mind, we see that there is a natural map
\[
\Phi: D_1\star\CE(D_2,\dots,D_{n+1}) \to \CE(D_1,\dots,D_{n+1})
\]
given by\footnote{The reader should keep in mind that an element $f$ in the domain is a continuous function on $[0,1]$ whose values are in turn (certain) continuous functions from $\Delta^{(n)}$ to the tensor product $D_1\otimes_{\max}\dots\otimes_{\max} D_{n+1}$.}  
\[
\Phi(f)(t_1,\vec{t}) = \begin{cases} f(1)(\vec{t}) &,\quad t_1=0 \\ f(1-t_1)\big( \frac{\vec{t}}{1-t_1} \big) &, \quad t_1\neq 0\end{cases}
\]
for $(t_1,\vec{t})\in\Delta^{(n+1)}$.
It is a simple exercise to see that this is a well-defined isomorphism.
This shows that it makes sense to view the \cstar-algebra $\CE(D_1,\dots,D_n)$ as the $n$-fold join $D_1\star\dots\star D_n$.
We can also observe that this isomorphism is natural in each \cstar-algebra, and therefore becomes equivariant as soon as we equip each \cstar-algebra $D_j$ with an action $\alpha^{(j)}$ of some group $G$.

Henceforth, we will in particular denote
\[
D^{\star n}:=\CE(D,n) \quad\text{and}\quad \alpha^{\star n}:=\CE(\alpha,n)
\]
for a unital \cstar-algebra $D$ and some group action $\alpha: G\curvearrowright D$.
\end{rem}

\begin{rem} \label{rem:multi-cones-extension}
By the definition of the join of two \cstar-algebras $D_1$ and $D_2$, there is a natural short exact sequence
\[
\xymatrix{
0 \ar[r] & \CC_0(0,1)\otimes D_1\otimes_{\max} D_2 \ar[r] & D_1\star D_2 \ar[r] & D_1\oplus D_2 \ar[r] & 0.
}
\]
Given some $n\geq 1$ and a unital \cstar-algebra $D$, we have $D^{\star n+1} \cong D\star(D^{\star n})$, and therefore a special case of the above yields the short exact sequence
\[
\xymatrix{
0 \ar[r] & \CC_0(0,1)\otimes D\otimes_{\max} D^{\star n} \ar[r] & D^{\star n+1} \ar[r] & D\oplus D^{\star n} \ar[r] & 0.
}
\]
Again by naturality, we note that this short exact sequence is automatically equivariant if we additionally equip $D$ with a group action. 
\end{rem}

We now come to the main observation about \cstar-dynamical systems arising in this fashion, which will be crucial in proving our main result:

\begin{lemma} \label{lem:multi-cones-absorption}
Let $G$ be a second-countable, locally compact group.
Let $A$ be a separable, unital \cstar-algebra with an action $\alpha: G\curvearrowright A$.
Suppose that $\gamma: G\curvearrowright\CD$ is a semi-strongly self-absorbing and unitarily regular action.
If $\alpha$ is $\gamma$-absorbing, then so is the action $\alpha^{\star n}: G\curvearrowright A^{\star n}$ for all $n\geq 2$.
\end{lemma}
\begin{proof}
This follows directly from \autoref{rem:multi-cones-extension} and \autoref{thm:ssa-ext} by induction.
\end{proof}

\begin{rem}
It ought to be mentioned that \autoref{lem:multi-cones-absorption} does not depend in any way on the fact that one considers the $n$-fold join over the same \cstar-algebra and the same action.
The analogous statement is valid for more general joins of the form 
\[
\alpha^{(1)}\star\dots\star\alpha^{(n)}: G\curvearrowright A_1\star\dots\star A_n
\] 
by virtually the same argument.

In fact, by putting in a bit more work, one could likely prove an equivariant version of \cite[Theorem 4.6]{HirshbergRordamWinter07} for $\CC_0(X)$-$G$-\cstar-algebras with $\dim(X)<\infty$ whose fibres absorb a given semi-strongly self-absorbing and unitarily regular action.
This would contain \autoref{lem:multi-cones-absorption} as a special case since the $G$-\cstar-algebra $A_1\star\dots\star A_n$ is in fact a $\CC(\Delta^{(n)})$-$G$-\cstar-algebra with each fibre being isomorphic to some finite tensor product of the $A_j$.
We will never need this level of generality within this paper, however.
\end{rem}


\section{Rokhlin dimension with commuting towers}

The following notion generalizes analogous definitions made in \cite{HirshbergWinterZacharias15, SzaboWuZacharias17, Gardella17, HirshbergSzaboWinterWu17}.

\begin{defi}[cf.\ {\cite[Definition 4.1]{HirshbergSzaboWinterWu17}}] \label{def:dimrokc}
Let $G$ be a second-countable, locally compact group. 
Let $\alpha: G\curvearrowright A$ be an action on a separable \cstar-algebra. 
\begin{enumerate}[label=\textup{(\roman*)},leftmargin=*]
\item Let $H\subset G$ be a closed, cocompact subgroup. 
The Rokhlin dimension of $\alpha$ with commuting towers relative to $H$, denoted $\dimrokc(\alpha, H)$, is the smallest natural number $d$ such that there exist equivariant c.p.c.\ order zero maps
\[
\phi^{(0)},\dots,\phi^{(d)}: (\CC(G/H), G\text{-shift})\to \big( F_{\infty,\alpha}(A), \tilde{\alpha}_\infty \big)
\]
with pairwise commuting ranges such that $\eins=\phi^{(0)}(\eins)+\dots+\phi^{(d)}(\eins)$.
\item If $\CS=(G_k)_k$ denotes a decreasing sequence of closed, cocompact subgroups, then we define the Rokhlin dimension of $\alpha$ with commuting towers relative to $\CS$ via
\[
\dimrokc(\alpha,\CS) = \sup_{k\in\IN}~\dimrokc(\alpha, G_k).
\]
\item Let $N\subset G$ be any closed, normal subgroup.
The Rokhlin dimension of $\alpha$ with commuting towers relative to $N$ is defined as
\[
\dimrokc(\alpha,N) := \sup\set{ \dimrokc(\alpha, H) \mid H\subseteq G \text{ closed, cocompact, } N\subseteq H }.
\]
\item Lastly, the Rokhlin dimension of $\alpha$ with commuting towers is defined as
\[
\dimrokc(\alpha) := \dimrokc(\alpha, \set{1}) = \sup\set{ \dimrokc(\alpha, H) \mid H\subseteq G \text{ closed, cocompact} }.
\]
\end{enumerate}
\end{defi}

We note that, even though the second half of \autoref{def:dimrokc} always makes sense, these concepts are not expected to be of any practical use when $G$ (or $G/N$) is not assumed to have enough closed cocompact subgroups, or to admit at least some residually compact approximation.

\begin{nota} \label{nota:procp-completion}
Let $G$ be a second-countable, locally compact group.
Given a decreasing sequence $\CS=(G_k)_k$ of closed, cocompact subgroups, we will denote
\[
G/\CS = \lim_{\longleftarrow} G/G_k.
\]
This is a metrizable, compact space\footnote{This construction generalizes the profinite completion of a discrete residually finite group along a chosen separating sequence of normal subgroups of finite index.}, which carries a natural continuous $G$-action induced by the left $G$-shift on each building block $G/G_k$; in particular we will call the resulting action also just the $G$-shift and denote it by 
\[
\sigma^\CS: G\curvearrowright G/\CS.
\]
In the sequel, we will adopt the perspective of the associated $G$-\cstar-dynamical system, which is given as the equivariant inductive limit
\[
\CC(G/\CS) = \lim_{\longrightarrow} \CC(G/G_k).
\]
We will moreover consider $\CC(G/S)^{\star n}$ for $n\geq 2$.
With some abuse of terminology, we will use the term ``$G$-shift'' also to refer to the canonical action on this \cstar-algebra (or the underlying space) that is induced by the $n$-fold tensor products of the $G$-shift on each fibre.
\end{nota}

\begin{lemma} \label{lem:dimrokc-eq}
Let $G$ be a second-countable, locally compact group. 
Let $\alpha: G\curvearrowright A$ be an action on a separable \cstar-algebra. 
Let $\CS=(G_k)_k$ be a decreasing sequence of closed, cocompact subgroups.
Let $d\geq 0$ be some natural number.
Then the following are equivalent:
\begin{enumerate}[label=\textup{(\roman*)},leftmargin=*]
\item $\dimrokc(\alpha,\CS)\leq d$; \label{lem:dimrokc-eq:1}
\item there exist equivariant c.p.c.\ order zero maps
\[
\phi^{(0)},\dots,\phi^{(d)}: (\CC(G/\CS), G\textup{-shift})\to \big( F_{\infty,\alpha}(A), \tilde{\alpha}_\infty \big)
\]
with pairwise commuting ranges such that $\eins=\phi^{(0)}(\eins)+\dots+\phi^{(d)}(\eins)$;\label{lem:dimrokc-eq:2}
\item there exists a unital $G$-equivariant $*$-homomorphism 
\[
\big( \CC(G/\CS)^{\star (d+1)}, G\textup{-shift} \big) \to \big( F_{\infty,\alpha}(A), \tilde{\alpha}_\infty \big);
\] \label{lem:dimrokc-eq:3}
\item the first-factor embedding
\[
\id_A\otimes\eins: (A,\alpha) \to \big( A\otimes\CC(G/\CS)^{\star (d+1)}, \alpha\otimes (G\textup{-shift}) \big)
\] 
is $G$-equivariantly sequentially split.
\label{lem:dimrokc-eq:4}
\end{enumerate}
\end{lemma}
\begin{proof}
The equivalence \ref{lem:dimrokc-eq:1}$\Leftrightarrow$\ref{lem:dimrokc-eq:2} follows from a standard reindexing trick such as Kirchberg's $\eps$-test \cite[Lemma A.1]{Kirchberg04}, using the equivariant inductive limit structure of $\CC(G/\CS)$ as pointed out in \autoref{nota:procp-completion}.
We will leave the details to the reader.

The equivalence \ref{lem:dimrokc-eq:2}$\Leftrightarrow$\ref{lem:dimrokc-eq:3} is a direct consequence of \autoref{prop:multi-cones-property} and \autoref{rem:multi-cones-join}, and the equivalence \ref{lem:dimrokc-eq:3}$\Leftrightarrow$\ref{lem:dimrokc-eq:4} is a direct consequence of \cite[Lemma 4.2]{BarlakSzabo16ss}.
\end{proof}

The purpose of this section is to prove the following theorem, which can be regarded as the main result of the paper.
Some of its non-trivial applications will be discussed in the subsequent sections.
See in particular \autoref{cor:dimrokc-cor} for a possibly more accessible special case of this theorem.

\begin{theorem} \label{thm:dimrokc-main}
Let $G$ be a second-countable, locally compact group and $N\subset G$ a closed, normal subgroup. 
Denote by $\pi_N: G\to G/N$ the quotient map. 
Let $\CS_1=(H_k)_k$ be a residually compact approximation of $G/N$, and set $G_k=\pi_N^{-1}(H_k)$ for all $k\in\IN$ and $\CS_0=(G_k)_k$.
Let $A$ be a separable \cstar-algebra and $\CD$ a strongly self-absorbing \cstar-algebra. 
Let $\alpha: G\curvearrowright A$ be an action and $\gamma: G\curvearrowright\CD$ a semi-strongly self-absorbing, unitarily regular action. 
Suppose that for the restrictions to the $N$-actions, we have $\alpha|_N \cc (\alpha\otimes\gamma)|_N$. 
If
\[
\asdim(\square_{\CS_1} H_1)<\infty\quad\text{and}\quad  \dimrokc(\alpha, \CS_0)<\infty,
\]
then $\alpha\cc\alpha\otimes\gamma$.
\end{theorem}

We note that \autoref{Thm-A} is a direct consequence of this result.
The hypothesis that $G/N$ has some discrete, normal, residually finite, cocompact subgroup admitting a box space with finite asymptotic dimension means that there is choice for $\CS_1$ as required by the above statement.
The hypothesis that $\alpha$ has finite Rokhlin dimension with commuting towers means that the value $\dimrok^c(\alpha,\CS_0)$ has a finite uniform upper bound, for any possible choice of $\CS_1$.
Hence the statement of \autoref{Thm-A} follows.

The proof of \autoref{thm:dimrokc-main} will occupy the rest of this section.
The first and most difficult step is to convince ourselves of a very special case of \autoref{thm:dimrokc-main}, which involves the technical preparation below and from Section \ref{sec:box}. 

For convenience, we isolate the following lemma, which is a consequence of \autoref{prop:cont-ell}, the Winter--Zacharias structure theorem for order zero maps, along with the Choi--Effros lifting theorem; see \cite[Section 3]{WinterZacharias09} and \cite{ChoiEffros76}.

\begin{lemma} \label{F(A)-lift}
Let $G$ be a second-countable, locally compact group. Let $A$ be a separable \cstar-algebra and $B$ a separable, unital and nuclear \cstar-algebra. 
Let $\alpha: G\curvearrowright A$ and $\beta: G\curvearrowright B$ be two actions. 
Let $\kappa: (B,\beta)\to \big( A_{\infty,\alpha}, \alpha_\infty \big)$ be an equivariant c.p.c.\ order zero map. 
Then $\kappa$ can be represented by a sequence of c.p.c.\ maps $\kappa_n: B\to A$ satisfying:
\begin{enumerate}[label=\textup{(\alph*)},leftmargin=*]
\item $\| \kappa_n(xy)\kappa(\eins)-\kappa_n(x)\kappa_n(y) \|\to 0$
\item $\dst\max_{g\in K}~\| (\kappa_n\circ\gamma_g)(x) - (\alpha_g\circ\kappa_n)(x)\|\to 0$
\end{enumerate}
for all $x,y\in B$ and compact subsets $K\subset G$.
\end{lemma}

The proof of the following is based on a standard reindexing trick.
In the short proof below, precise references are provided for completeness, although we note that this might not be the most elegant or direct way to show these statements.

\begin{lemma} \label{lem:reindexation}
Let $G$ be a second-countable, locally compact group.
Suppose that $\alpha: G\curvearrowright A$, $\beta: G\curvearrowright B$, and $\gamma: G\curvearrowright\CD$ are actions on separable \cstar-algebras.
Assume furthermore that $\CD$ is unital, that $\gamma$ is semi-strongly self-absorbing, and that $\beta\cc\beta\otimes\gamma$.
\begin{enumerate}[label=\textup{(\roman*)},leftmargin=*]
\item Suppose that there exists an equivariant $*$-homomorphism $(A,\alpha)\to (B,\beta)$ that is $G$-equivariantly sequentially split. Then $\alpha\cc\alpha\otimes\gamma$. \label{lem:reindexation:1}
\item Suppose that $B$ is unital and that there exists an equivariant and unital $*$-homomorphism from $(B,\beta)$ to $(F_{\infty,\alpha},\tilde{\alpha}_\infty)$.
Then $\alpha\cc\alpha\otimes\gamma$. \label{lem:reindexation:2}
\end{enumerate}
\end{lemma}
\begin{proof}
\ref{lem:reindexation:1}: By \autoref{thm:equi-McDuff}, the statement $\beta\cc\beta\otimes\gamma$ is equivalent to the equivariant first-factor embedding
\[
\id_B\otimes\eins: (B,\beta)\to (B\otimes\CD,\beta\otimes\gamma)
\]
being sequentially split.
Let $\phi: (A,\alpha)\to (B,\beta)$ be sequentially split.
By \cite[Proposition 3.7]{BarlakSzabo16ss}, the composition 
$\phi\otimes\eins_\CD=(\id_B\otimes\eins_\CD)\circ\phi$ is also sequentially split.
However, we also have $\phi\otimes\eins_\CD = (\phi\otimes\id_\CD)\circ(\id_A\otimes\eins_\CD)$, which implies that $\id_A\otimes\eins_\CD$ is also sequentially split.
This implies the claim that $\alpha\cc\alpha\otimes\gamma$.

\ref{lem:reindexation:2}: By \cite[Lemma 4.2]{BarlakSzabo16ss}, it follows that the embedding
\[
\id_A\otimes\eins_B: (A,\alpha)\to (A\otimes_{\max} B,\alpha\otimes\beta)
\]
is sequentially split. 
Since we assumed that $\beta$ is $\gamma$-absorbing, so is $\alpha\otimes\beta$, and so the claim arises as a special case of \ref{lem:reindexation:1}.
\end{proof}

The following is a special case of \autoref{thm:dimrokc-main}, as the process of tensorially stabilizing any action $\alpha: G\curvearrowright A$ with $(\CC(G/\CS),\sigma^\CS)$ causes the Rokhlin dimension relative to $\CS$ to collapse to zero by definition.
This explains why the statement below makes no explicit reference to Rokhlin dimension.
Its proof is by far the most technical part of this paper:

\begin{lemma} \label{lem:main-technical}
Let $G$ be a second-countable, locally compact group and $N\subset G$ a closed, normal subgroup. 
Denote by $\pi_N: G\to G/N$ the quotient map. 
Let $\CS_1=(H_k)_k$ be a residually compact approximation of $G/N$, and set $G_k=\pi_N^{-1}(H_k)$ for all $k\in\IN$ and $\CS_0=(G_k)_k$.
Let $A$ be a separable \cstar-algebra and $\CD$ a strongly self-absorbing \cstar-algebra. 
Let $\alpha: G\curvearrowright A$ be an action and $\gamma: G\curvearrowright\CD$ a semi-strongly self-absorbing, unitarily regular action. 
Suppose that for the restrictions to the $N$-actions, we have $\alpha|_N \cc (\alpha\otimes\gamma)|_N$. 
If $\asdim(\square_{\CS_1} H_1)<\infty$, then the $G$-action
\[
\sigma^{\CS_0}\otimes\alpha : G\curvearrowright \CC(G/\CS_0)\otimes A
\]
is $\gamma$-absorbing.
\end{lemma}
\begin{proof}
Set $d:=\asdim(\square_{\CS_1} H_1)$.
Let
\[
\tilde{\kappa}: (\CD,\gamma|_N)\to\big( F_{\infty,\alpha|_N}(A), \tilde{\alpha}_\infty|_N \big)
\]
be an $N$-equivariant, unital $*$-homomorphism.
Using \cite[Example 4.4, Proposition 4.5]{Szabo18ssa2}, we may choose an equivariant c.p.c.\ order zero map
\[
\kappa: (\CD,\gamma|_N)\to \big( A_{\infty,\alpha|_N}\cap A' , \alpha_\infty|_N \big)
\] 
that lifts $\tilde{\kappa}$.

Consider a sequence of c.p.c.\ maps $\kappa_n: B\to A$ lifting $\kappa$ as in \autoref{F(A)-lift}. 
Let us choose finitely many subsequences $\kappa_n^{(l)}: B\to A$ of the maps $\kappa_n$ for $l=0,\dots,d$ so that, using \autoref{F(A)-lift}, each sequence $\kappa_n^{(l)}$ has the following properties for all $a\in A$, $b,b_1,b_2\in \CD$ and compact sets $L\subseteq N$: 
\begin{equation} \label{e:kappa-1}
\|[\kappa_n^{(l)}(b),a]\|\to 0;
\end{equation}
\begin{equation} \label{e:kappa-2}
\|  \kappa_n^{(l)}(b_1b_2)\kappa_n^{(l)}(\eins)-\kappa_n^{(l)}(b_1)\kappa_n^{(l)}(b_2)\|\to 0;
\end{equation}
\begin{equation} \label{e:kappa-3} 
\|\big(\kappa_n^{(l)}(\eins)-\eins\big)\cdot a\|\to 0;
\end{equation}
\begin{equation} \label{e:kappa-4}
\max_{r\in L}~\| (\kappa_n^{(l)}\circ\gamma_r)(b) - (\alpha_r\circ\kappa_n^{(l)})(b)\|\to 0;
\end{equation}
and additionally one has for every compact set $K\subseteq G$ that
\begin{equation} \label{e:kappa-5}
\max_{g\in K}~ \|[\kappa_n^{(l_1)}(b_1),(\alpha_g\circ\kappa_n^{(l_2)})(b_2)]\|\to 0\quad\text{for all } l_1\neq l_2.
\end{equation}

Let now $\eps>0$ be a fixed parameter and $1_G\in K\subseteq G$ a fixed compact set. 
Apply \autoref{lem:decay} and find $k$ and compactly supported functions $\mu^{(0)},\dots,\mu^{(d)}\in\CC_c(G/N)$, so that for every $l=0,\dots,d$, we have
\begin{equation} \label{e:decay-1}
\supp(\mu^{(l)})\cap \supp(\mu^{(l)})\cdot h = \emptyset \quad\text{for all}~h\in H_k\setminus\set{1};
\end{equation}
\begin{equation} \label{e:decay-2}
\sum_{l=0}^d~ \sum_{h\in H_k} \mu^{(l)}(\pi_N(g)h) = 1 \quad\text{for all}~g\in G;
\end{equation}
\begin{equation} \label{e:decay-3}
\|\mu^{(l)}(\pi_N(g)\cdot\_)-\mu^{(l)}\|_\infty \leq \eps \quad\text{for all}~g\in K\cup K^{-1}.
\end{equation}
The group $H_k$ is discrete, so we may choose a cross-section $\sigma: H_k\to G_k=\pi_N^{-1}(H_k)\subseteq G$.
For each $l=0,\dots,d$, consider the sequence of c.p.c.\ maps
\[
\Theta^{(l)}_n: \CD\to \CC_b(G, A)
\]
given by
\begin{equation} \label{e:Theta}
\Theta^{(l)}_n(b)(g) = \sum_{h\in H_k} \mu^{(l)}(\pi_N(g)h)\cdot (\alpha_{g\sigma(h)}\circ\kappa_n^{(l)}\circ\gamma_{g\sigma(h)}^{-1})(b).
\end{equation} 
This sum is well-defined because the compact support of the function $\mu^{(l)}$ on $G/N$ meets a set of the form $\pi_N(g)\cdot H_k$ at most once according to \eqref{e:decay-1}.

We wish to show that given an element $b\in\CD$, the functions $\Theta_n^{(l)}(b)$ are approximately $G_k$-periodic on large compact sets.
This is so that we may apply \autoref{cor:cpc-perturb} in order to approximate the maps $\Theta^{(l)}_n$ by other maps going into $\CC(G/G_k, A)$. 

Let $K_{H_k}\subseteq G_k$ and $K_G\subseteq G$ be two compact sets. As $H_k$ is discrete, we observe two facts. First, there exists a compact set $K_N\subseteq N$ and a finite set $1\in F_k\fin H_k$ with 
\begin{equation} \label{e:K-überdeckung}
K_{H_k}\subset \sigma(F_k)\cdot K_N. 
\end{equation}
Second, by possibly enlarging $F_k$ if necessary, we may assume by \eqref{e:decay-1} that also
\begin{equation} \label{e:decay-positiv}
\mu^{(l)}(\pi_N(g)h)>0 \quad\text{implies}\quad h\in F_k \quad\text{for all }g\in K_G.
\end{equation}
Define also
\begin{equation} \label{eq:K-N-strich}
K_N' = \bigcup_{h_0, h\in F_k} \sigma(h_0)\cdot K_N\cdot\sigma(h_0^{-1}h)\sigma(h)^{-1} \subseteq N
\end{equation}
and
\begin{equation} \label{e:K-N-strich-strich}
K_N'' = \bigcup_{h\in F_k} \sigma(h)^{-1}\cdot K_N'\cdot \sigma(h) \subseteq N.
\end{equation}
As $N$ is a normal subgroup and $\sigma$ is a cross-section for the quotient map $\pi_N$, it follows that these are compact subsets in $N$.

We compute for all $l=0,\dots,d$, $b\in\CD$, $g\in K_G$, $h_0\in F_k$ and $r\in K_N'$ that
\[
\renewcommand{\arraystretch}{1.8}
\begin{array}{cl}
\multicolumn{2}{l}{ \|(\alpha_{g\sigma(h_0)}\circ\kappa^{(l)}_n\circ\gamma_{g\sigma(h_0)}^{-1})(b)-(\alpha_{gr\sigma(h_0)}\circ\kappa^{(l)}_n\circ\gamma_{gr\sigma(h_0)}^{-1})(b)\| } \\
=& \|(\alpha_{\sigma(h_0)}\circ\kappa^{(l)}_n\circ\gamma_{\sigma(h_0)}^{-1})(\gamma_g^{-1}(b))-(\alpha_{r\sigma(h_0)}\circ\kappa^{(l)}_n\circ\gamma_{r\sigma(h_0)}^{-1})(\gamma_g^{-1}(b))\| \\
=& \|(\alpha_{\sigma(h_0)}\circ\kappa^{(l)}_n\circ\gamma_{\sigma(h_0)}^{-1})(\gamma_g^{-1}(b)) \\
& -(\alpha_{\sigma(h_0)}\circ\alpha_{\sigma(h_0)^{-1}r\sigma(h_0)}\circ\kappa^{(l)}_n\circ\gamma^{-1}_{\sigma(h_0)^{-1}r\sigma(h_0)}\circ\gamma_{\sigma(h_0)}^{-1})(\gamma_g^{-1}(b))\| \\
\stackrel{\eqref{e:K-N-strich-strich}}{\leq} & \dst \max_{g\in K_G}~\max_{s\in K_N''}~\|(\alpha_s\circ\kappa^{(l)}_n\circ\gamma_s^{-1})(\gamma_{g\sigma(h_0)}^{-1}(b))-\kappa^{(l)}_n(\gamma_{g\sigma(h_0)}^{-1}(b)) \| \\
\stackrel{\eqref{e:kappa-4}}{\longrightarrow} & 0. \qquad(\text{uniformly on } K_G, K_N')
\end{array}
\]
It thus follows for all $l=0,\dots,d$, $b\in\CD$, $g\in K_G$, $h_0\in F_k$ and $r\in K_N$ that
\[
\renewcommand{\arraystretch}{1.8}
\begin{array}{cl}
\multicolumn{2}{l}{ \|\Theta^{(l)}_n(b)(g)-\Theta^{(l)}_n(b)(g\sigma(h_0)r)\| } \\
\stackrel{\eqref{e:Theta}, \eqref{e:decay-positiv}}{=}\hspace{-3mm}& \dst \Big\| \sum_{h_1\in F_k} \mu^{(l)}(\pi_N(g)h_1)\cdot (\alpha_{g\sigma(h_1)}\circ\kappa_n^{(l)}\circ\gamma_{g\sigma(h_1)}^{-1})(b) \\
& \dst - \sum_{h_2\in h_0^{-1}F_k} \mu^{(l)}(\pi_N(g)h_0h_2)\cdot (\alpha_{g\sigma(h_0)r\sigma(h_2)}\circ\kappa_n^{(l)}\circ\gamma_{g\sigma(h_0)r\sigma(h_2)}^{-1})(b) \Big\| \\
=& \dst \Big\| \sum_{h_1\in F_k} \mu^{(l)}(\pi_N(g)h_1)\cdot (\alpha_{g\sigma(h_1)}\circ\kappa_n^{(l)}\circ\gamma_{g\sigma(h_1)}^{-1})(b) \\
& \dst - \sum_{h_2\in F_k} \mu^{(l)}(\pi_N(g)h_2)\cdot (\alpha_{g\sigma(h_0)r\sigma(h_0^{-1}h_2)}\circ\kappa_n^{(l)}\circ\gamma_{g\sigma(h_0)r\sigma(h_0^{-1}h_2)}^{-1})(b) \Big\| \\
\stackrel{\eqref{e:decay-1}}{=}& \dst \max_{h\in F_k}~ \|(\alpha_{g\sigma(h)}\circ\kappa_n^{(l)}\circ\gamma_{g\sigma(h)}^{-1})(b) \\
& -(\alpha_{g\sigma(h_0)r\sigma(h_0^{-1}h)}\circ\kappa_n^{(l)}\circ\gamma_{g\sigma(h_0)r\sigma(h_0^{-1}h)})(b) \| \\
=& \dst \max_{h\in F_k}~ \|(\alpha_{g\sigma(h)}\circ\kappa_n^{(l)}\circ\gamma_{g\sigma(h)}^{-1})(b) \\
& -(\alpha_{g\sigma(h_0)r\sigma(h_0^{-1}h)\sigma(h)^{-1}\sigma(h)}\circ\kappa_n^{(l)}\circ\gamma_{g\sigma(h_0)r\sigma(h_0^{-1}h)\sigma(h)^{-1}\sigma(h)}^{-1})(b) \| \\
\stackrel{\eqref{eq:K-N-strich}}{=}& \dst \max_{h\in F_k}~\max_{s\in K_N'}~ \| (\alpha_{g\sigma(h)}\circ\kappa_n^{(l)}\circ\gamma_{g\sigma(h)}^{-1})(b) - (\alpha_{gs\sigma(h)}\circ\kappa_n^{(l)}\circ\gamma_{gs\sigma(h)}^{-1})(b) \| \\
\stackrel{}{\longrightarrow} & 0. \qquad(\text{uniformly on } K_G, K_N)
\end{array}
\]
Here we have used \eqref{e:decay-1} in the third equality in the sense that $\mu^{(l)}(\pi_N(g)h)$ is non-zero for a unique element $h\in F_k$.
By \eqref{e:K-überdeckung} we get for all $b\in \CD$ that
\[
\renewcommand{\arraystretch}{1.8}
\begin{array}{cl}
\multicolumn{2}{l}{ \dst \max_{g\in K_G}~\max_{t\in K_{H_k}}~\|\Theta^{(l)}_n(b)(g)-\Theta^{(l)}_n(b)(gt)\| } \\
\stackrel{\eqref{e:K-überdeckung}}{\leq} & \dst \max_{g\in K_G}~\max_{h_0\in F_k}~\max_{r\in K_N}~\|\Theta^{(l)}_n(b)(g)-\Theta^{(l)}_n(b)(g\sigma(h_0)r)\| \\
\stackrel{}{\longrightarrow} & 0.
\end{array}
\]
Since $K_G\subseteq G$ and $K_{H_k}\subseteq G_k$ were arbitrary compact sets, we are in the position to apply \autoref{cor:cpc-perturb}.
As $\CD$ is separable, it follows for every $l=0,\dots,d$ that there exists a sequence of c.p.c.\ maps
\[
\Psi^{(l)}_n: B\to\CC(G/G_k, A) 
\]
so that for every compact set $K_G\subseteq G$ and $b\in\CD$, we have
\begin{equation} \label{e:Theta-Psi}
\max_{g\in K_G}~ \|\Psi^{(l)}_n(b)(gG_k)-\Theta^{(l)}_n(b)(g)\|\to 0.
\end{equation}

We now wish to show that these c.p.c.\ maps are approximately equivariant with regard to $\gamma$ and $\sigma^{G_k}\otimes\alpha$, where $\sigma^{G_k}$ is the $G$-action on $\CC(G/G_k)$ induced by the left-translation of $G$ on $G/G_k$.
 
Let us fix a compact set $K_{G}\subseteq G$ as above. 
Without loss of generality, let us assume that it is large enough so that the quotient map $G\to G/G_k$ is still surjective when restricted to $K_{G}$. 
Given $b\in \CD$, set
\begin{equation} \label{e:rho(b)}
\rho_n(b) = \max_{l=0,\dots,d}~\max_{g\in K^{-1}K_{G}}~ \|\Psi^{(l)}_n(b)(gG_k)-\Theta^{(l)}_n(b)(g)\|.
\end{equation}
Note that by an elementary compactness argument, it follows from \eqref{e:Theta-Psi} that for every compact set $J\subset \CD$, we have
\begin{equation} \label{e:rho-est}
\max_{b\in J}~\rho_n(b) \to 0.
\end{equation}
Let $t\in K$, $g\in K_{G}$ and $b\in \CD$ with $\norm b\norm\leq 1$. 
Then
\[
\begin{array}{cl}
\multicolumn{2}{l}{ (\sigma^{G_k}_t\otimes\alpha_t)\Big( (\Psi_n^{(l)})(b) \Big)(gG_k) } \\ 
=& \alpha_t\big(\Psi^{(l)}_n(b)(t^{-1}gG_k) \big) \\
\stackrel{\eqref{e:rho(b)}}{=}_{\makebox[0pt]{\footnotesize\hspace{1mm}$\rho_n(b)$}} & \alpha_t\big( \Theta_n^{(l)}(b)(t^{-1}gG_k) \big)\\
\stackrel{\eqref{e:Theta}}{=}& \dst \sum_{h\in H_k} \mu^{(l)}(\pi_N(t^{-1}g)h)\cdot (\alpha_{g\sigma(h)}\circ\kappa_n^{(l)}\circ\gamma_{t^{-1}g\sigma(h)}^{-1})(b) \\
\stackrel{\eqref{e:decay-1},\eqref{e:decay-3}}{=}_{\makebox[0pt]{\footnotesize\hspace{-10.5mm}$\eps$}} & \dst \sum_{h\in H_k} \mu^{(l)}(\pi_N(g)h)\cdot (\alpha_{g\sigma(h)}\circ\kappa_n^{(l)}\circ\gamma_{t^{-1}g\sigma(h)}^{-1})(b) \\
\phantom{--}\stackrel{\eqref{e:rho(b)},\eqref{e:Theta}}{=}_{\makebox[0pt]{\footnotesize\hspace{-1mm}$\rho_n(\gamma_t(b))$}}\phantom{--} & \Psi_n^{(l)}(\gamma_t(b))(gG_k).
\end{array}
\]
Note that as $K_{G}$ contains a representative for every $G_k$-orbit in $G$, these approximations carry over to the $\|\cdot\|_\infty$-norm of the involved functions.
Using \eqref{e:rho-est}, we obtain for all $b\in\CD$ with $\norm b\norm\leq 1$ that
\begin{equation} \label{e:Psi-approx-equivariant}
\limsup_{n\to\infty}~\max_{t\in K}~\|(\sigma^{G_k}_t\otimes\alpha_t)(\Psi_n^{(l)})(b)-(\Psi_n^{(l)}\circ\gamma_t)(b)\| ~\leq~ \eps.
\end{equation}
Next, we wish to show that for $l_1\neq l_2$, the c.p.c.\ maps $\Psi^{(l_1)}_n$ and $\Psi^{(l_2)}_n$ have approximately commuting ranges as $n\to\infty$. 
Let $g_1,g_2\in K_{G}$ and $b\in \CD$ with $\norm b\norm\leq 1$ be given. 
Then we compute 
\[
\begin{array}{cl}
\multicolumn{2}{l}{ \big\| \big[ \Psi^{(l_1)}_n(b)(g_1G_k), \Psi^{(l_2)}_n(b)(g_2G_k) \big] \big\| } \\
\stackrel{\eqref{e:rho(b)}}{=}_{\makebox[0pt]{\footnotesize\hspace{2mm}$4\rho_n(b)$}} & \big\| \big[ \Theta^{(l_1)}_n(b)(g_1), \Theta^{(l_2)}_n(b)(g_2) \big] \big\| \\
\hspace{-3mm}\stackrel{\eqref{e:decay-1},\eqref{e:Theta},\eqref{e:decay-positiv}}{\leq}\hspace{-3mm}& \dst \max_{h_1, h_2\in F_k}~\Big\| \Big[ (\alpha_{g_1\sigma(h_1)}\circ\kappa_n^{(l_1)}\circ\gamma_{g_1\sigma(h_1)}^{-1})(b), (\alpha_{g_2\sigma(h_2)}\circ\kappa_n^{(l_2)}\circ\gamma_{g_2\sigma(h_2)}^{-1})(b) \Big] \Big\| \\
=& \dst \max_{h_1, h_2\in F_k}~\Big\| \Big[ (\kappa_n^{(l_1)}\circ\gamma_{g_1\sigma(h_1)}^{-1})(b), (\alpha_{\sigma(h_1)^{-1}g_1^{-1}g_2\sigma(h_2)}\circ\kappa_n^{(l_2)}\circ\gamma_{g_2\sigma(h_2)}^{-1})(b) \Big] \Big\|
\end{array}
\]
In particular, we obtain for every contraction $b\in\CD$ that
\begin{equation} \label{e:Psi-approx-commuting}
\begin{array}{cl}
\multicolumn{2}{l}{ \dst\max_{g_1,g_2\in K_G}~ \big\| \big[ \Psi^{(l_1)}_n(b)(g_1G_k), \Psi^{(l_2)}_n(b)(g_2G_k) \big] \big\| }  \\
\leq & \dst \max_{g_1,g_2\in K_G}~\max_{h_1, h_2\in F_k}~\Big\| \Big[ (\kappa_n^{(l_1)}\circ\gamma_{g_1\sigma(h_1)}^{-1})(b), (\alpha_{\sigma(h_1)^{-1}g_1^{-1}g_2\sigma(h_2)}\circ\kappa_n^{(l_2)}\circ\gamma_{g_2\sigma(h_2)}^{-1})(b) \Big] \Big\| \\
& +4\rho_n(b)\\
\hspace{-3mm}\stackrel{\eqref{e:rho-est},\eqref{e:kappa-5}}{\longrightarrow}\hspace{-3mm} & 0.
\end{array}
\end{equation}
Here we have used that the convergence in \eqref{e:kappa-5} automatically holds uniformly when quantifying over $b_1, b_2$ belonging to some compact subset in $\CD$, in this case
\[
b_1, b_2\in \set{ \gamma^{-1}_g(b) \mid g\in K_G\cdot\sigma(F_k) }.
\]

In exactly the same fashion, one also computes
\begin{equation} \label{e:Psi-approx-central}
\big\| \big[ \Psi^{(l)}_n(b), a \big] \big\| \longrightarrow 0
\end{equation} 
for all $l=0,\dots,d$, $b\in\CD$, and $a\in A$, by using \eqref{e:kappa-1} in place of \eqref{e:kappa-5}.

Next, we wish to show that for each $l=0,\dots,d$, the c.p.c.\ maps $\Psi^{(l)}_n$ behave approximately like order zero maps.
Let $g\in K_{G}$. Choose the unique element $h_0\in F_k$ with $\mu^{(l)}(\pi_N(g)h_0)>0$. 
Then it follows for every $b_1, b_2\in \CD$ that
\[
\begin{array}{ll}
\multicolumn{2}{l}{ \Theta^{(l)}_n(b_1)(g)\cdot \Theta^{(l)}_n(b_2)(g) } \\
=& \mu^{(l)}(\pi_N(g)h_0)^2\cdot (\alpha_{g\sigma(h_0)}\circ\kappa_n^{(l)}\circ\gamma_{g\sigma(h_0)}^{-1})(b_1)\cdot(\alpha_{g\sigma(h_0)}\circ\kappa_n^{(l)}\circ\gamma_{g\sigma(h_0)}^{-1})(b_2) \\
=& \mu^{(l)}(\pi_N(g)h_0)^2\cdot \alpha_{g\sigma(h_0)}\Big( (\kappa_n^{(l)}\circ\gamma_{g\sigma(h_0)}^{-1})(b_1) \cdot (\kappa_n^{(l)}\circ\gamma_{g\sigma(h_0)}^{-1})(b_1) \Big) .
\end{array}
\]
It follows from this calculation that
\[
\begin{array}{cl}
\multicolumn{2}{l}{ \big\| \Theta^{(l)}_n(b_1)\cdot \Theta^{(l)}_n(b_2) - \Theta^{(l)}_n(b_1b_2)\cdot \Theta^{(l)}_n(\eins)  \big\|_{\infty, K_G} } \\
\leq & \dst \max_{s\in K_{G_k}\cdot \sigma(F_k)} \big\| (\kappa_n^{(l)}\circ\gamma_{s}^{-1})(b_1) \cdot (\kappa_n^{(l)}\circ\gamma_{s}^{-1})(b_1) \\
& \phantom{----}- (\kappa_n^{(l)}\circ\gamma_{s}^{-1})(b_1b_2) \cdot (\kappa_n^{(l)}\circ\gamma_{s}^{-1})(\eins) \big\| \\
\stackrel{\eqref{e:decay-2},\eqref{e:decay-3}}{\longrightarrow} & 0.
\end{array}
\]
As $K_G$ contains a representative of every $G_k$-orbit in $G$, it follows from \eqref{e:Theta-Psi} that
\begin{equation} \label{e:Psi-approx-order-zero}
\big\| \Psi^{(l)}_n(b_1)\cdot \Psi^{(l)}_n(b_2) - \Psi^{(l)}_n(b_1b_2)\cdot \Psi^{(l)}_n(\eins)  \big\| \longrightarrow 0
\end{equation}
for every $b_1, b_2\in\CD$.

Next, we wish to show that the completely positive sum $\sum_{l=0}^d \Psi^{(l)}_n$ behaves approximately like a u.c.p.\ map upon multiplication with an element of $\eins\otimes A$, as $n\to\infty$. 
Let $g\in K_{G}$. 
We have 
\[
\begin{array}{cl}
\multicolumn{2}{l}{ \dst \Theta^{(l)}_n(\eins)(g) } \\
\stackrel{\eqref{e:Theta},\eqref{e:K-überdeckung}}{=}& \dst \sum_{h\in F_k} \mu^{(l)}(\pi_N(g) h)\cdot (\alpha_{g\sigma(h)}\circ\kappa_n^{(l)}\circ\gamma_{g\sigma(h)}^{-1})(\eins) \\
=& \dst \sum_{h\in F_k} \mu^{(l)}(\pi_N(g) h)\cdot (\alpha_{g\sigma(h)}\circ\kappa_n^{(l)})(\eins).
\end{array}
\]
It follows for all $a\in A$ that
\[
\begin{array}{cl}
\multicolumn{2}{l}{ \dst \max_{g\in K_G}~\Big\| \Big( \eins - \sum_{l=0}^d\Theta^{(l)}_n(\eins)(g) \Big) \cdot a \Big\| } \\
\leq & \dst \max_{g\in K_G}~ (d+1)\cdot \max_{l}~\max_{h\in F_k}~\|(\alpha_{\sigma(h)}\circ\kappa^{(l)}_n)(\eins)-\kappa^{(l)}_n(\eins)\| \\
& \dst + \Big\|  \Big(\eins - \sum_{l=0}^d \sum_{h\in F_k} \mu^{(l)}(\pi_N(g) h)\cdot (\alpha_{g}\circ\kappa_n^{(l)})(\eins) \Big)\cdot a \Big\| \\
\stackrel{\eqref{e:decay-2}}{=}& \dst \max_{g\in K_G}~ (d+1)\cdot \max_{l}~\max_{h\in F_k}~\|(\alpha_{\sigma(h)}\circ\kappa^{(l)}_n)(\eins)-\kappa^{(l)}_n(\eins)\| \\
& \dst + \Big\|  \Big( \sum_{l=0}^d \sum_{h\in F_k} \mu^{(l)}(\pi_N(g) h)\cdot \big( \eins - (\alpha_{g}\circ\kappa_n^{(l)})(\eins) \big) \Big)\cdot a \Big\| \\
\stackrel{\eqref{e:decay-1}}{\leq} & \dst \max_{g\in K_G}~\Big( (d+1)\cdot \max_{l}~\max_{h\in F_k}~\|(\alpha_{\sigma(h)}\circ\kappa^{(l)}_n)(\eins)-\kappa^{(l)}_n(\eins)\| \\
& \dst + (d+1)\cdot \max_{l}~ \| \big(\eins-\kappa^{(l)}_n(\eins) \big)\cdot\alpha_g^{-1}(a) \| \Big) \\
\stackrel{\eqref{e:kappa-3},\eqref{e:kappa-4}}{\longrightarrow}\hspace{-2mm} & 0.
\end{array}
\]
Since $K_G$ contains every $G_k$-orbit in $G$, it follows from \eqref{e:Theta-Psi} that
\begin{equation} \label{e:Psi-approx-jointly-unital}
\Big\| \Big( \eins - \sum_{l=0}^d\Psi^{(l)}_n(\eins) \Big) \cdot (\eins\otimes a) \Big\| \to 0 \quad\text{for all}~a\in A.
\end{equation}

Let us now summarize everything we have obtained so far. 
The c.p.c.\ maps $\Psi^{(l)}_n: \CD\to\CC(G/G_k,A)$, for $l=0,\dots,d$ and $n\in\IN$ satisfy the following properties for all $b,b_1,b_2\in\CD$ and $a\in A$:
\begin{equation} \label{e:psi-1}
\| [\Psi_n^{(l)}(b),\eins\otimes a] \| \longrightarrow 0;
\end{equation}
\begin{equation} \label{e:psi-2}
\limsup_{n\to\infty}~ \max_{t\in K}~ \| \big( (\sigma^{G_k}\otimes\alpha)_t\circ\Psi^{(l)}_n \big)(b)-(\Psi^{(l)}_n\circ\gamma_t)(b)\|\leq \eps;
\end{equation}
\begin{equation} \label{e:psi-3}
\| [\Psi^{(l_1)}_n(b),\Psi^{(l_2)}_n(b)] \| \longrightarrow 0 \quad\text{for all } l_1\neq l_2;
\end{equation}
\begin{equation} \label{e:psi-4}
\| \Psi^{(l)}_n(b_1)\cdot\Psi^{(l)}_n(b_2)-\Psi^{(l)}_n(b_1b_2)\cdot\Psi^{(l)}_n(\eins)  \| \longrightarrow 0;
\end{equation}
\begin{equation} \label{e:psi-5}
\Big\| \Big( \eins- \sum_{l=0}^d \Psi^{(l)}_n(\eins) \Big)\cdot \eins\otimes a \Big\| \longrightarrow 0.
\end{equation}
Note that $k$, and thus the codomain of $\Psi^{(l)}_n$, had to be chosen depending on $\eps$ and $K\subseteq G$.
However, we have canonical (equivariant) inclusions $\CC(G/G_k, A)\subseteq\CC(G/\CS_0, A)$, which we may compose our maps with.
It is then clear that the same properties as in \eqref{e:psi-1} up to \eqref{e:psi-5} hold, where we replace the action $\sigma^{G_k}: G\curvearrowright\CC(G/G_k)$ by $\sigma^{\CS_0}: G\curvearrowright\CC(G/\CS_0)$. 

Since $A$ and $B$ are separable and $G$ is second-countable, we can let the tolerance $\eps$ go to zero, let the set $K\subseteq G$ get larger and apply a diagonal sequence argument.
Putting the appropriate choices of c.p.c.\ maps into a single sequence, we can thus obtain c.p.c.\ maps
\[
\psi^{(l)}: B\to \big( \CC(G/\CS_0)\otimes A \big)_{\infty},\quad l=0,\dots,d
\]
that satisfy the following properties for all $g\in G$, $a\in A$, and $b, b_1, b_2\in\CD$:
\begin{equation} \label{e:psii-1}
[\psi^{(l)}(b),\eins\otimes a]=0;
\end{equation}
\begin{equation} \label{e:psii-2}
(\sigma^{\CS_0}\otimes\alpha)_g\circ\psi^{(l)}=\psi^{(l)}\circ\gamma_g;
\end{equation}
\begin{equation} \label{e:psii-3}
[\psi^{(l_1)}(b),\psi^{(l_2)}(b)]=0 \quad\text{for all}~l_1\neq l_2;
\end{equation}
\begin{equation} \label{e:psii-4}
 \psi^{(l)}(b_1)\cdot\psi^{(l)}(b_2) = \psi^{(l)}(b_1b_2)\cdot\psi^{(l)}(\eins);
\end{equation}
\begin{equation} \label{e:psii-5}
\Big( \eins- \sum_{l=0}^d \psi^{(l)}(\eins) \Big)\cdot \eins\otimes a = 0.
\end{equation}
Since $\gamma: G\curvearrowright \CD$ is point-norm continuous, \eqref{e:psii-2} implies that the image of each map $\psi^{(l)}$ is in the continuous part $\big( \CC(G/\CS_0)\otimes A \big)_{\infty,\sigma^{\CS_0}\otimes\alpha}$.
In fact it is in the relative commutant of $\eins\otimes A$ by \eqref{e:psii-1}, but then also automatically in the relative commutant of all of $\CC(G/\CS_0)\otimes A$.
This allows us to define equivariant maps
\[
\zeta^{(l)}: \CD\to F_{\infty,\sigma^{\CS_0}\otimes\alpha}\big( \CC(G/\CS_0)\otimes A \big), \quad \zeta^{(l)}(b)=\psi^{(l)}(b)+\ann(\CC(G/\CS_0)\otimes A)
\]
for all $l=0,\dots,d$.
Then \eqref{e:psii-3} implies that these maps have commuting ranges, \eqref{e:psii-4} implies that they are c.p.c.\ order zero, and finally \eqref{e:psii-5} implies the equation $\dst\sum_{l=0}^d \zeta^{(l)}(\eins)=\eins$.

By virtue of \autoref{prop:multi-cones-property} and \autoref{rem:multi-cones-join}, this gives rise to a unital equivariant $*$-homomorphism
\[
\big( \CD^{\star (d+1)}, \gamma^{\star(d+1)} \big) \to \big( F_{\infty,\sigma^{\CS_0}\otimes\alpha}\big( \CC(G/\CS_0)\otimes A \big), (\widetilde{\sigma^{\CS_0}\otimes\alpha})_\infty \big).
\]
As $\gamma$ is unitarily regular, it follows from \autoref{lem:multi-cones-absorption} that $\gamma^{\star(d+1)}$ is a $\gamma$-absorbing action.
Applying \autoref{lem:reindexation} yields that $\sigma^{\CS_0}\otimes\alpha$ is $\gamma$-absorbing, which finishes the proof.
\end{proof} 

Now we are in a position to prove \autoref{thm:dimrokc-main}:

\begin{proof}[Proof of \autoref{thm:dimrokc-main}]
Let $\alpha: G\curvearrowright A$ and $\gamma: G\curvearrowright\CD$ be the two actions as in the assumption.
Let also $N\subset G$, $H_k\subset G/N$, and $G_k\subset G$ be subgroups as specified in the statement, and denote $\CS_1=(H_k)_k$ as a sequence of subgroups in $G/N$, and $\CS_0=(G_k)_k$ as a sequence of subgroups in $G$.

Suppose $\asdim(\square_{\CS_1} H_1)<\infty$ and $s:=\dimrokc(\alpha,\CS_0)<\infty$.
Using the latter, \autoref{lem:dimrokc-eq}\ref{lem:dimrokc-eq:4} implies that the equivariant embedding 
\[
\id_A\otimes\eins: (A,\alpha) \to \big( A\otimes\CC(G/\CS_0)^{\star (s+1)}, \alpha\otimes (G\textup{-shift}) \big)
\] 
is $G$-equivariantly sequentially split.
By \autoref{lem:reindexation}, in order to show that $\alpha$ is $\gamma$-absorbing, it suffices to show that the $G$-\cstar-algebra $A\otimes\CC(G/\CS_0)^{\star (s+1)}$ is $\gamma$-absorbing.
We will show this via induction on $s$.

For $s=0$, the claim is precisely \autoref{lem:main-technical}, and in particular it holds because we assumed $\asdim(\square_{\CS_1} H_1)<\infty$.

Given $s\geq 1$, assume that the claim holds for $s-1$.
It follows by \autoref{rem:multi-cones-extension} that there is an extension of $G$-\cstar-algebras of the form
\[
\xymatrix{
0 \ar[r] & J^{(s)} \ar[r] & A\otimes\CC(G/\CS_0)^{\star (s+1)} \ar[r] & Q^{(s)} \ar[r] & 0,
}
\]
where
\[
J^{(s)} = A\otimes \CC_0(0,1)\otimes \CC(G/\CS_0)\otimes \CC(G/\CS_0)^{\star s}
\]
and
\[
Q^{(s)} = A\otimes\Big( \CC(G/\CS_0)\oplus \CC(G/\CS_0)^{\star s} \Big).
\]
By the induction hypothesis, both the kernel and the quotient of this extension are $\gamma$-absorbing $G$-\cstar-algebras, and therefore so is the middle by \autoref{thm:ssa-ext}.
This finishes the induction step and the proof.
\end{proof}

\begin{rem}
We remark that the statement of the main result holds verbatim for cocycle actions instead of genuine actions.
Note that the concept of Rokhlin dimension makes sense for cocycle actions with the same definition, since there is still a natural genuine action induced on the central sequence algebra.
If $(\alpha,w): G\curvearrowright A$ is a cocycle action on a separable \cstar-algebra, then $(\alpha\otimes\id_\CK,w\otimes\eins): G\curvearrowright A\otimes\CK$ is cocycle conjugate to a genuine action by the Packer--Raeburn stabilization trick \cite{PackerRaeburn89}.
Since both Rokhlin dimension and absorption of a semi-strongly self-absorbing action are invariants under stable (cocycle) conjugacy, the statement of \autoref{thm:dimrokc-main} follows for cocycle actions.
\end{rem}


\section{Some applications}
\label{sec:applications}

Let us now discuss some immediate applications of the main result.
First we wish to point out that the following result arises as a special case.

\begin{cor} \label{cor:dimrokc-cor}
Let $G$ be a second-countable, locally compact group. 
Let $\CS=(H_n)_n$ be a residually compact approximation consisting of normal subgroups of $G$ with 
\[
\asdim(\square_\CS H_1)<\infty.
\]
Let $A$ be a separable \cstar-algebra and $\CD$ a strongly self-absorbing \cstar-algebra with $A\cong A\otimes\CD$. 
Let $\alpha: G\curvearrowright A$ be an action with 
\[
\dimrokc(\alpha, \CS)<\infty.
\]
Then $\alpha\vscc\alpha\otimes\gamma$ for all semi-strongly self-absorbing actions $\gamma: G\curvearrowright\CD$.
\end{cor}
\begin{proof}
Let $\gamma: G\curvearrowright\CD$ be a semi-strongly self-absorbing action.
Since $\CD\cong\CD\otimes\CZ$ by \cite{Winter11}, we may replace $\gamma$ with $\gamma\otimes\id_\CZ$ for the purpose of showing the claim, as $\gamma\otimes\id_\CZ$ is again semi-strongly self-absorbing and every $(\gamma\otimes\id_\CZ)$-absorbing action is $\gamma$-absorbing. 
So let us simply assume $\gamma\cc\gamma\otimes\id_\CZ$.
By \autoref{rem:unitary-reg}, we may thus assume that $\gamma$ is unitarily regular.
The claim then follows directly from \autoref{thm:dimrokc-main} applied to the case $N=\set{1}$.
Note that one automatically has absorption with respect to very strong cocycle conjugacy by virtue of \autoref{thm:equi-McDuff}\ref{equi-McDuff5}.
\end{proof}

Note that the results below in part refer to Rokhlin dimension without commuting towers, as defined in \cite[Section 5]{SzaboWuZacharias17}.
For the Rokhlin dimension zero case, the commuting tower assumption is vacuous.

\begin{example} \label{ex:uni-UHF}
Let $\CQ$ denote the universal UHF algebra.
Let $\Gamma$ be a countable, discrete group and $H\subset\Gamma$ a normal subgroup with finite index.
There exists a strongly self-absorbing action $\gamma: G\curvearrowright\CQ$ with $\dimrok(\gamma,H)=0$.
\end{example}
\begin{proof}
Such an action is constructed as part of \cite[Remark 10.8]{SzaboWuZacharias17}.
Namely, consider the left-regular representation $\lambda^{G/H}: G/H\to\CU(M_{|G:H|})$, consider the quotient map $\pi_H: G\to G/H$, and define
\[
\gamma_g=\id_\CQ\otimes\bigotimes_\IN \ad(\lambda^{G/H}(\pi_H(g)))
\]
as an action on $\CQ\cong\CQ\otimes M_{[G:H]}^{\otimes\infty}$.
As the diagonal embedding $\CC(G/H)\subset M_{[G:H]}$ is equivariant, it follows that $\dimrok(\gamma,H)=0$.
By \cite[Proposition 6.3]{Szabo18ssa2}, such an action is strongly self-absorbing.
\end{proof}

This in turn has the following consequence regarding the dimension-reducing effect of strongly self-absorbing \cstar-algebras.

\begin{cor} \label{cor:reducing-dimrok}
Let $\Gamma$ be a countable, discrete, residually finite group that has some box space with finite asymptotic dimension.
Let $\alpha: \Gamma\curvearrowright A$ be an action on a separable \cstar-algebra with $\dimrokc(\alpha)<\infty$.
\begin{enumerate}[label=\textup{(\arabic*)},leftmargin=*]
\item If $A\cong A\otimes\CQ$, then $\dimrok(\alpha)=0$. \label{cor:reducing-dimrok:1}
\item If $A\cong A\otimes\CZ$, then $\dimrok(\alpha)\leq 1$.  \label{cor:reducing-dimrok:2}
\end{enumerate}
\end{cor}
\begin{proof}
\ref{cor:reducing-dimrok:1} follows directly from \autoref{ex:uni-UHF} and \autoref{cor:dimrokc-cor}.

\ref{cor:reducing-dimrok:2}: We have $\alpha\cc\alpha\otimes\id_\CZ$, and there exist two c.p.c.\ order zero maps $\psi_0,\psi_1: \CQ\to\CZ_\infty\cap \CZ'$ with $\psi_0(\eins)+\psi_1(\eins)=\eins$; see \cite[Section 5]{MatuiSato14UHF} or \cite[Section 6]{SatoWhiteWinter15}.
Consider two sequences $\psi_{0,n}, \psi_{1,n}: \CQ\to\CZ$ of c.p.c.\ maps lifting $\psi_0$ and $\psi_1$. 

By \ref{cor:reducing-dimrok:1}, $\alpha\otimes\id_\CQ$ has Rokhlin dimension zero.
Given any subgroup $H\subset \Gamma$ with finite index, we can find c.p.c.\ order zero maps $\CC(\Gamma/H)\to A\otimes\CQ$, which are approximately equivariant, have approximately central image, and so that the image of the unit acts approximately like a unit on finite sets.
Once we compose such maps with $\id_A\otimes\psi_{i,n}$ for $i=0,1$ and large enough $n$, we may obtain two c.p.c.\ maps $\kappa_0, \kappa_1: C(\Gamma/H)\to A\otimes\CZ$, which are approximately equivariant, have approximately central image, and so that the element $\kappa_0(\eins)+\kappa_1(\eins)$ approximately acts like a unit on a given finite set in $A\otimes\CZ$.
But this is what is required by $\dimrok(\alpha)=\dimrok(\alpha\otimes\id_\CZ)\leq 1$; we leave the finer details to the reader as the proof is quite standard.
\end{proof}

\begin{rem}
The reason why the proof of \autoref{cor:reducing-dimrok}\ref{cor:reducing-dimrok:2} does not yield $\dimrokc(\alpha)\leq 1$ is due to the fact that the two order zero maps $\psi_0, \psi_1: \CQ\to\CZ_\infty$ can never have commuting ranges.
Indeed, this would imply the existence of a unital $*$-homomorphism $\CQ\to\CZ_\infty$ via \autoref{lem:multi-cones-absorption}, so it is impossible.
More concretely, \cite[Example 3.32]{GardellaHirshbergSantiago17} exhibits an example of a $\IZ_2$-action $\alpha$ on a Kirchberg algebra with $\dimrok(\alpha)=1$ and $\dimrokc(\alpha)=2$.
\end{rem}

\begin{cor} \label{cor:reducing-dimrok-X}
Let $\Gamma$ be a discrete, finitely generated, virtually nilpotent group.
Let $X$ be a compact metrizable space with finite covering dimension, and $\alpha: \Gamma\curvearrowright X$ a free action by homeomorphisms.
Then one has
\[
\dimrok\big( \alpha\otimes\id_{\CQ}: \Gamma\curvearrowright\CC(X)\otimes\CQ \big)=0
\]
and
\[
\dimrok\big( \alpha\otimes\id_\CZ: \Gamma\curvearrowright\CC(X)\otimes\CZ \big)\leq 1.
\]
\end{cor}
\begin{proof}
By \cite[Corollary 7.5]{SzaboWuZacharias17}, the action $\alpha: \Gamma\curvearrowright\CC(X)$ has finite Rokhlin dimension.\footnote{Strictly speaking, only the nilpotent case is proved there. The virtually nilpotent case follows from independent work of Bartels \cite[Section 1]{Bartels17}.}
Since the underlying \cstar-algebra is abelian, the claim follows from \autoref{cor:reducing-dimrok}.
\end{proof}


\section{Multi-flows on strongly self-absorbing Kirchberg algebras}
\label{sec:multiflows}

In this section, we shall study actions of $\IR^k$ on certain \cstar-algebras satisfying an obvious notion of the Rokhlin property.

\begin{nota}
For $k\geq 2$, we will refer to a continuous action of $\IR^k$ on a \cstar-algebra as a {\it multi-flow}.
Let $(e_j)_{1\leq j\leq k}$ be the standard basis of $\IR^k$.
Given $\alpha: \IR^k\curvearrowright A$, we will denote the {\it generating flows} $\alpha^{(j)}:\IR\curvearrowright A$ given by $\alpha^{(j)}_t = \alpha_{te_j}$, for $j=1,\dots,k$.
We then have $\alpha^{(j)}_{t_j}\circ\alpha^{(i)}_{t_i}=\alpha^{(i)}_{t_i}\circ\alpha^{(j)}_{t_j}$ for all $i,j=1,\dots,k$ and all $t_i, t_j\in\IR$.
We will also denote by $\alpha^{(\msout{j})}: \IR^{k-1}\curvearrowright A$ the action generated by the flows $(\alpha^{(i)})_{i\neq j}$.
We remark that $\alpha^{(j)}$ reduces naturally to a flow on the fixed point algebra $A^{{\alpha}^{(\msout{j})}}$.
\end{nota}

\begin{defi} \label{def:Rp-mf}
Let $A$ be a separable \cstar-algebra and $\alpha: \IR^k\curvearrowright A$ an action.
We say that $\alpha$ has the Rokhlin property, if $\dimrok(\alpha,p\IZ^k)=0$ for all $p\greater 0$.
\end{defi}

\begin{rem}
An obvious question regarding the above definition is whether this is the same as $\dimrok(\alpha)=0$ when $k\geq 2$, especially because this appears to be (a priori) much more difficult to check.
Nevertheless, this turns out to be case. 
Instead of giving a detailed proof here, let us just roughly sketch the basic idea.

The condition $\dimrok(\alpha)=0$ in the sense of \autoref{def:dimrokc} amounts to checking $\dimrok(\alpha,H)=0$ for every closed cocompact subgroup $H\subset\IR^k$, or in other words finding approximately equivariant unital embeddings from $\CC(\IR^k/H)$ into the central sequence algebra of $A$.
This only gets easier when we make $H$ smaller, so we may assume without loss of generality that $H$ is discrete.
Since $H$ is a free abelian group and is cocompact in $\IR^k$, it has a $\IZ$-basis $e_1,\dots,e_k\in H$.
We may approximate these elements by $f_1,\dots,f_k\in\IQ^k$, which are linearly independent over $\IQ$ and span another subgroup $H'$.
By using for example \autoref{lem:periodic-expect} we can then obtain approximately multiplicative and equivariant u.c.p.\ maps $\CC(\IR^k/H)\to \CC(\IR^k/H')$.
By the properties of central sequence algebras, we may thus assume without loss of generality that in fact $H\subseteq\IQ^k$.
Now the same argument as in \cite[Example 3.19]{SzaboWuZacharias17} allows one to see that $H$ contains a finite index subgroup of the form $n\IZ^k$ for some $n\in\IN$.
In summary, we obtain $\dimrok(\alpha, H)=0$ for arbitrary $H$ when we assume the Rokhlin property in the sense of \autoref{def:Rp-mf}.
\end{rem}

\begin{rem}
In the case of flows, i.e., the case $k=1$ above, \autoref{def:Rp-mf} coincides with Kishimoto's notion of the Rokhlin property from \cite{Kishimoto96_R}.
Let us for now denote by $\sigma^T: \IR\curvearrowright\CC(\IR/T\IZ)$ the action induced by the $\IR$-shift.
\end{rem}

\begin{prop} \label{prop:multi-Rokhlin}
Let $A$ be a separable \cstar-algebra and $\alpha: \IR^k\curvearrowright A$ an action.
The following are equivalent:
\begin{enumerate}[label={\textup{(\roman*)}},leftmargin=*]
\item $\alpha$ has the Rokhlin property; \label{multi-Rokhlin:1}
\item for every $j=1,\dots,k$ and every $p\greater 0$, there exists a unitary
\[
u\in F_{\infty,\alpha}(A)^{\tilde{\alpha}^{(\msout{j})}_\infty} \quad\text{such that}\quad \tilde{\alpha}^{(j)}_{\infty,t}(u)=e^{ipt}u,\quad t\in\IR;
\] \label{multi-Rokhlin:2}
\item for every $j=1,\dots,k$ and every $T\greater 0$, there exists an equivariant and unital $*$-homomorphism
\[
\big( \CC(\IR/T\IZ), \sigma^T\big) \longrightarrow \big( F_{\infty,\alpha}(A)^{\tilde{\alpha}^{(\msout{j})}_\infty}, \tilde{\alpha}^{(j)}_{\infty} \big).
\] \label{multi-Rokhlin:3}
\end{enumerate}
\end{prop}
\begin{proof}
\ref{multi-Rokhlin:1}$\LRa$\ref{multi-Rokhlin:3}: Let $T\greater 0$.
One has a canonical equivariant isomorphism
\[
\big( \CC(\IR^k/T\IZ^k), \IR^k\textup{-shift} \big) \cong \big( \CC(\IR/T\IZ)^{\otimes k}, \sigma^{T,1}\otimes\dots\otimes\sigma^{T,k} \big),
\]
where $\sigma^{T,j}$ is the $\IR^k$-action on $\CC(\IR/T\IZ)$ where only the $j$-th component acts by the $\IR$-shift.
By definition, $\alpha$ having the Rokhlin property means that for every $T\greater 0$ the dynamical system on the left embeds into $\big( F_{\infty,\alpha}(A), \tilde{\alpha}_\infty \big)$.
So in particular, when \ref{multi-Rokhlin:1} holds, then one also obtains an embedding of $\big( \CC(\IR/T\IZ), \sigma^{T,j} \big)$ for every $j=1,\dots,k$, which implies \ref{multi-Rokhlin:3}.
Conversely, when \ref{multi-Rokhlin:3} holds, then for all $T\greater 0$ one has an embedding of $\big( \CC(\IR/T\IZ), \sigma^{T,j} \big)$ into $\big( F_{\infty,\alpha}(A), \tilde{\alpha}_\infty \big)$ for all $j=1,\dots,k$.
By applying a standard reindexing argument in the central sequence algebra, one may assume that these embeddings have pairwise commuting ranges for all $j=1,\dots,k$.
Therefore one obtains an embedding of the \cstar-dynamical system given by the tensor product of all $\big( \CC(\IR/T\IZ), \sigma^{T,j} \big)$, which we have seen to be the same as the dynamical system $\big( \CC(\IR^k/T\IZ^k), \IR^k\textup{-shift} \big)$.
In particular this implies \ref{multi-Rokhlin:1}.

\ref{multi-Rokhlin:2}$\LRa$\ref{multi-Rokhlin:3}: This follows directly from functional calculus.
A unitary $u$ as in \ref{multi-Rokhlin:2} gives rise to a unital equivariant $*$-homomorphism 
\[
\phi_u: \big( \CC(\IR/\frac{2\pi}{p}\IZ), \sigma^{\frac{2\pi}{p}}\big) \longrightarrow \big( F_{\infty,\alpha}(A)^{\tilde{\alpha}^{(\msout{j})}_\infty}, \tilde{\alpha}^{(j)}_{\infty} \big),\quad \phi_u(f)=f(u).
\]
Conversely, whenever $\phi$ is an arbitrary homomorphism between these two dynamical systems, then $u=\phi([t+\frac{2\pi}{p}\IZ\mapsto e^{ipt}])$ yields a unitary as required by \ref{multi-Rokhlin:2}.
\end{proof}

\begin{rem} \label{rem:Rk-box}
We note that for $G=\IR^k$, the sequence $H_n=(n!)\cdot\IZ^k$ yields a residually compact approximation in the sense of \autoref{def:res-cp-approx}.
Now it is well-known that $\square_{(H_n)_n} \IZ^k$ has finite asymptotic dimension $k$; see either \cite[Sections 2+3]{SzaboWuZacharias17} or better yet \cite{DelabieTointon17}.
In particular, \autoref{cor:dimrokc-cor} is applicable to $\IR^k$-actions that have finite Rokhlin dimension with commuting towers, and more specifically it is applicable to $\IR^k$-actions with the Rokhlin property.
\end{rem}

The following is the main result of this section.

\begin{theorem} \label{thm:multiflows}
Let $\CD$ be a strongly self-absorbing Kirchberg algebra.
Let $k\geq 1$ be a given natural number.
Then all continuous $\IR^k$-actions on $\CD$ with the Rokhlin property are semi-strongly self-absorbing and are mutually (very strongly) cocycle conjugate.
\end{theorem}

The approach for proving this result, at least in the way presented here, uses the theory of semi-strongly self-absorbing actions in a crucial way.
In such dynamical systems, one has a very strong control over certain (approximately central) unitary paths, which together with the Rokhlin property allows one to obtain a relative cohomology-vanishing-type statement.
This will be used to deduce inductively that the actions in the statement of \autoref{thm:multiflows} have approximately $\IR^k$-inner flip.
The desired uniqueness for such actions is then achieved by combining this fact with \autoref{cor:dimrokc-cor}, which is a special case of our main result, in a suitable way.

\begin{example}[see \cite{BratteliKishimotoRobinson07}]
Denote by $s_1, s_2, \dots$ the generators of the Cuntz algebra $\CO_\infty$.
Define a quasi-free flow $\gamma^0: \IR\curvearrowright\CO_\infty$ via $\gamma^0_t(s_1)=e^{2\pi i t}s_1$, $\gamma^0_t(s_2)=e^{-2\pi i\sqrt{2} t}s_2$, and $\gamma^0_t(s_j)=s_j$ for $j\geq 3$.
Then $\gamma^0$ has the Rokhlin property by \cite[Theorem 1.1]{BratteliKishimotoRobinson07}.

In particular, given $k\geq 1$ and any strongly self-absorbing Kirchberg algebra $\CD$, the action
\[
\id_\CD\otimes\big( \underbrace{\gamma^0\times\dots\times\gamma^0}_{k \text{ times}} \big): \IR^k\curvearrowright \CD\otimes\CO_\infty^{\otimes k} \cong \CD
\]
is a ($k$-)multi-flow with the Rokhlin property on $\CD$, and is in fact (very strongly) cocycle conjugate to every other one by \autoref{thm:multiflows}.
\end{example}

Let us now implement the strategy outlined above step by step.
We begin with the aforementioned cohomology-vanishing-type statement, which involves minimal assumptions about the underlying \cstar-algebras but otherwise very strong assumptions about the existence of certain unitary paths, which will naturally appear in our intended setup later.

\begin{lemma} \label{lem:hilfslemma}
Let $A$ be a separable unital \cstar-algebra.
Let $k \geq 1$ and let $\alpha: \IR^k\curvearrowright A$ be a continuous action with the Rokhlin property, and fix some $j\in\set{1,\dots,k}$. 

For every $\eps\greater 0$, $L\greater 0$ and $\CF\fin A$, there exists a $T\greater 0$ and $\CG\fin A$ with the following property:

If $\{w_t\}_{t\in\IR}\subset\CU(A)$ is any $\alpha^{(j)}$-1-cocycle satisfying
\[
\max_{a\in\CF}~\max_{0\leq t\leq T}~ \norm [w_t,a]\norm \leq\eps,\quad \max_{0\leq t\leq T}~ \max_{\vec{r}\in [0,1]^{k-1}}~ \norm w_t-\alpha^{(\msout{j})}_{\vec{r}}(w_t) \norm\leq\eps, 
\]
and moreover there exists some continuous path of unitaries $u: [0,1]\to\CU(A)$ with
\[
u(0)=\eins,\quad u(1)=w_{-T},\quad \ell(u)\leq L,
\]
\[
\max_{0\leq t\leq 1}~\max_{\vec{r}\in [0,1]^{k-1}} \norm u(t)-\alpha^{(\msout{j})}_{\vec{r}}(u(t)) \norm \leq \eps,
\]
\[
\max_{0\leq t\leq 1}~\max_{a\in\CG}~ \norm [u(t),a] \norm \leq \eps,
\]
then there exists a unitary $v\in \CU(A)$ 
satisfying
\[
\max_{0\leq t\leq 1} \norm w_t-v\alpha^{(j)}_t(v^*) \norm \leq 3\eps,
\]
\[
\quad \max_{a\in\CF}~ \norm [v,a] \norm \leq 3\eps,
\]
\[
\quad \max_{\vec{r}\in [0,1]^{k-1}}~ \norm v-\alpha^{(\msout{j})}_{\vec{r}}(v) \norm \leq 3\eps.
\]
\end{lemma}
\begin{proof}
Let $T\greater 0$ and note that we have fixed $j\in\set{1,\dots,k}$ by assumption.
By some abuse of notation, let us view $\sigma^T$ as the $\IR^k$-action on $\CC(\IR/T\IZ)$ such that the $j$-th coordinate acts as the $\IR$-shift and all the other components act trivially.
In this way, any $*$-homomorphism as in \autoref{prop:multi-Rokhlin}\ref{multi-Rokhlin:3} can be viewed as an $\IR^k$-equivariant $*$-homomorphism from $\CC(\IR/T\IZ)$ to $F_{\infty,\alpha}(A)$.
In particular, denote such a homomorphism by $\theta$.
We can then obtain a commutative diagram of $\IR^k$-equivariant $*$-homomorphisms via
\begin{equation} \label{eq:hilfslemma:1}
\xymatrix{
(A,\alpha) \ar[rr] \ar[rd]_{d\mapsto \eins\otimes d} && (A_{\infty,\alpha},\alpha_\infty) \\
& \big( \CC(\IR/T\IZ)\otimes A, \sigma^T\otimes\alpha \big)  \ar[ru]_{f\otimes d \mapsto \theta(f)\cdot d} &
}
\end{equation}
We will keep this in mind for later.

Now let $\eps\greater 0$, $L\greater 0$ and $\CF\fin A$ be as in the statement.
Without loss of generality, we assume that $\CF$ consists of contractions.
We choose $T\greater \frac{L}{\eps}$ and $\CG\fin A$ any finite set of contractions containing $\CF$ that is $\eps/2$-dense in the compact subset
\begin{equation} \label{eq:hilfslemma:2}
\set{ \alpha^{(j)}_{-s}(a) \mid a\in\CF,\ 0\leq s\leq T}.
\end{equation}
We claim that these do the trick. 
We note that the rest of the proof below is almost identical to the proof of \cite[Theorem 2.1]{Kishimoto96_R} and \cite[Lemma 3.4]{Szabo17R}, respectively, except for some obvious modifications.

Assume that $\{w_t\}_{t\in\IR}\subset\CU(A)$ is an $\alpha^{(j)}$-1-cocycle satisfying
\begin{equation} \label{eq:hilfslemma:a-w-1}
\max_{a\in\CF}~\max_{0\leq t\leq T}~ \norm [w_t,a]\norm \leq\eps;
\end{equation}
\begin{equation} \label{eq:hilfslemma:a-w-2}
\max_{0\leq t\leq T}~ \max_{\vec{r}\in [0,1]^{k-1}}~ \norm w_t-\alpha^{(\msout{j})}_{\vec{r}}(w_t) \norm\leq\eps; 
\end{equation}
and moreover that there exists some continuous path of unitaries $u: [0,1]\to\CU(A)$ with
\begin{equation} \label{eq:hilfslemma:a-u-1}
u(0)=\eins,\quad u(1)=w_{-T},\quad \ell(u)\leq L;
\end{equation}
\begin{equation} \label{eq:hilfslemma:a-u-2}
\max_{0\leq t\leq 1}~\max_{\vec{r}\in [0,1]^{k-1}} \norm u(t)-\alpha^{(\msout{j})}_{\vec{r}}(u(t)) \norm \leq \eps;
\end{equation}
\begin{equation} \label{eq:hilfslemma:a-u-3}
\max_{0\leq t\leq 1}~\max_{a\in\CG}~ \norm [u(t),a] \norm \leq \eps.
\end{equation}
As $\ell(u)\leq L$, we may assume that $u$ is $L$-Lipschitz by passing to the arc-length parameterization if necessary.
We denote by $\kappa: [0,T]\to\CU(A)$ the path given by $\kappa_s=u(s/T)$, which is then Lipschitz with respect to the constant $L/T\leq \eps$.
Let us define a continuous path of unitaries $v: [0,T]\to\CU(A)$ via $v_s=w_s\alpha^{(j)}_s(\kappa_s)$.
Then by \eqref{eq:hilfslemma:a-u-1} it follows that $v(0)=v(T)=\eins$.
In particular, we may view $v$ as a unitary in $\CC(\IR/T\IZ)\otimes A$.

We have
\[
\begin{array}{ccl}
\dst\max_{a\in\CF}~ \norm [v,\eins\otimes a] \norm &=& \dst\max_{a\in\CF}~\max_{0\leq s\leq T}~ \norm [w_s\alpha^{(j)}_s(\kappa_s),a] \norm \\
&\leq& \dst\max_{a\in\CF}~\max_{0\leq s\leq T}~ \norm [w_s,a] \norm + \norm [\kappa_s,\alpha^{(j)}_{-s}(a)]\norm \\
&\stackrel{\eqref{eq:hilfslemma:a-w-1}}{\leq}& \dst\eps+\max_{a\in\CF}~\max_{0\leq s\leq T}~  \norm [\kappa_s,\alpha^{(j)}_{-s}(a)]\norm \\
&\stackrel{\eqref{eq:hilfslemma:2}}{\leq}& \dst 3\eps/2+\max_{b\in\CG}~ \norm [\kappa_s,b]\norm \\
&\stackrel{\eqref{eq:hilfslemma:a-u-3}}{\leq}& 5\eps/2.
\end{array}
\]
Moreover, we have
\[
\begin{array}{cl}
\multicolumn{2}{l}{ \dst \max_{\vec{r}\in [0,1]^{k-1}}~ \norm v-(\sigma^T\otimes\alpha)^{(\msout{j})}_{\vec{r}}(v) \norm }\\
=& \dst \max_{\vec{r}\in [0,1]^{k-1}}~ \norm v-(\id\otimes\alpha)^{(\msout{j})}_{\vec{r}}(v) \norm \\
=& \dst \max_{\vec{r}\in [0,1]^{k-1}}~ \max_{0\leq s\leq T}~ \norm v_s-\alpha^{(\msout{j})}_{\vec{r}}(v_s) \norm \\
=& \dst \max_{\vec{r}\in [0,1]^{k-1}}~ \max_{0\leq s\leq T}~ \norm w_s\alpha^{(j)}_s(\kappa_s)-\alpha^{(\msout{j})}_{\vec{r}}(w_s\alpha^{(j)}_s(\kappa_s)) \norm \\
=& \dst \max_{\vec{r}\in [0,1]^{k-1}}~ \max_{0\leq s\leq T}~ \norm w_s\alpha^{(j)}_s(\kappa_s)-\alpha^{(\msout{j})}_{\vec{r}}(w_s)\cdot \alpha^{(j)}_s(\alpha^{(\msout{j})}_{\vec{r}}(\kappa_s)) \norm \\
\leq& \dst \max_{\vec{r}\in [0,1]^{k-1}}~ \max_{0\leq s\leq T}~ \norm w_s-\alpha^{(\msout{j})}_{\vec{r}}(w_s) \norm + \norm \kappa_s- \alpha^{(\msout{j})}_{\vec{r}}(\kappa_s) \norm \\
\stackrel{\eqref{eq:hilfslemma:a-w-2},\eqref{eq:hilfslemma:a-u-2}}{\leq}& 2\eps.
\end{array}
\]
Lastly, let us fix $t\in [0,1]$ and $s\in [0,T]$.
If $s\geq t$, then we compute
\[
\begin{array}{cl}
\multicolumn{2}{l}{ \big( v(\sigma^T\otimes\alpha)^{(j)}_t(v^*) \big)(s) }\\
=& w_s\alpha^{(j)}_s(\kappa_s) \cdot \alpha^{(j)}_t\big( \alpha^{(j)}_{s-t}(\kappa_{s-t}^*) w_{s-t}^* \big) \\
=& w_s\cdot \alpha^{(j)}_s\big( \kappa_s\kappa_{s-t}^* \big) \alpha^{(j)}_t(w_{s-t}^*) \\
\stackrel{\eqref{eq:hilfslemma:a-u-1}}{=}_{\makebox[0pt]{\footnotesize\hspace{-2.5mm}$\eps$}} & w_s\alpha^{(j)}_t(w_{s-t}^*) \ = \ w_t.
\end{array}
\]
On the other hand, if $s\leq t$, then in particular $s\leq 1$ and $T-1\leq T+s-t\leq T$, and we compute
\[
\begin{array}{cl}
\multicolumn{2}{l}{ \big( v(\sigma^T\otimes\alpha)^{(j)}_t(v^*) \big)(s) }\\
=& w_s\alpha^{(j)}_s(\kappa_s)\cdot \alpha^{(j)}_t\big( \alpha^{(j)}_{T+s-t}(\kappa_{T+s-t}^*) w_{T+s-t}^* \big) \\
=& w_s\alpha^{(j)}_s(\kappa_s)\cdot \alpha^{(j)}_{T+s}(\kappa_{T+s-t}^*) \alpha^{(j)}_t(w_{T+s-t}^*) \\
\stackrel{\eqref{eq:hilfslemma:a-u-1}}{=}_{\makebox[0pt]{\footnotesize\hspace{-1.5mm}$2\eps$}} & w_s\cdot\eins\cdot\alpha^{(j)}_{T+s}(w_{-T}^*)\alpha^{(j)}_t(w_{T+s-t}^*) \\
=& w_s\alpha^{(j)}_{s}(w_{T})\alpha^{(j)}_t(w_{T+s-t}^*) \\
=& w_{T+s}\alpha^{(j)}_t(w_{T+s-t}^*)w_t^*\cdot w_t \ = \ w_t.
\end{array}
\]
Let us summarize what we have accomplished so far.
Starting from the existence of the $\alpha^{(j)}$-1-cocycle $\set{w_t}_{t\in\IR}$ and the unitary path $u$ with the prescribed properties, we have found a unitary $v\in\CU\big(\CC(\IR/T\IZ)\otimes A\big)$ satisfying
\begin{equation}
\max_{a\in\CF}~ \norm [v,\eins\otimes a] \norm \leq 5\eps/2;
\end{equation}
\begin{equation}
\max_{\vec{r}\in [0,1]^{k-1}}~ \norm v-(\sigma^T\otimes\alpha)^{(\msout{j})}_{\vec{r}}(v) \norm \leq 2\eps;
\end{equation}
\begin{equation}
\max_{0\leq t\leq 1}~ \norm w_t-v(\sigma^T\otimes\alpha)^{(j)}_t(v^*) \norm \leq 2\eps.
\end{equation}
By using the commutative diagram \eqref{eq:hilfslemma:1}, we may send $v$ into the sequence algebra of $A$, represent the resulting unitary by a sequence of unitaries in $A$, and then select a member of this sequence so that it will satisfy the properties in the claim with respect to the parameter $3\eps$.
This finishes the proof.
\end{proof}

Now record the following useful technical result about semi-strongly self-absorbing actions, which arises as a special case of \cite[Lemma 3.12]{Szabo18ssa2}:

\begin{lemma} \label{lem:ssa-ac-paths}
Let $G$ be a second-countable, locally compact group. 
Let $\CD$ be a separable, unital \cstar-algebra and $\gamma: G\curvearrowright\CD$ a semi-strongly self-absorbing action. 
For every $\eps>0$, $\CF\fin\CD$ and compact set $K\subset G$, there exist $\delta>0$ and $\CG\fin\CD$ with the following property:

Suppose that $u: [0,1]\to\CU(\CD)$ is a unitary path satisfying
\[
u(0)=\eins,\quad \max_{0\leq t\leq 1}~ \max_{g\in K}~ \norm u(t)-\gamma_g(u(t)) \norm \leq \delta,
\] 
and
\[
\max_{a\in\CG}~ \|[u(1), a]\|\leq\delta.
\]
Then there exists a unitary path $w: [0,1]\to\CU(\CD)$ satisfying
\[
w(0)=\eins,\quad w(1)= u(1),
\]
\[
\max_{g\in K}~ \norm w(t)-\gamma_g(w(t)) \norm \leq \eps,
\]
\[
\max_{0\leq t\leq 1}~\max_{a\in\CF}~\|[w(t), a]\|\leq\eps.
\] 
Moreover, we may choose $w$ in such a way that
\[
\|w(t_1)-w(t_2)\| \leq \|u(t_1)-u(t_2)\|\quad\text{for all}\quad 0\leq t_1, t_2\leq 1.
\]
\end{lemma}

We are now ready to prove the main result of this section:

\begin{proof}[Proof of \autoref{thm:multiflows}]
We will prove this via induction in $k$.
For this purpose, we will include the case $k=0$, where the claim is true for trivial reasons.

Now let $k\geq 1$ and assume that the claim is true for actions of $\IR^{k-1}$.
We will then show that the claim is also true for actions of $\IR^k$.

\textbf{Step 1:}
Let $\alpha: \IR^k\curvearrowright\CD$ be an action with the Rokhlin property.
In a similar fashion as in \cite[Proposition 3.5]{Kishimoto02}, we shall show that $\alpha$ has approximately $\IR^k$-inner flip.

Set $B=\CD\otimes\CD$ and $\beta=\alpha\otimes\alpha$.
Denote by $\Sigma$ the flip automorphism on $B$, which is equivariant with regard to $\beta$.
Note that $\beta$ is still a $\IR^k$-action on a strongly self-absorbing Kirchberg algebra with the Rokhlin property.
The $\IR^{k-1}$-action $\alpha^{(\msout{k})}$ is semi-strongly self-absorbing by the induction hypothesis.
Applying \cite[Proposition 3.6]{Szabo18ssa2}, we find a sequence of unitaries $y_n, z_n\in\CU(B)$ satisfying
\begin{equation} \label{eq:multiflows:yz-1}
\max_{\vec{r}\in [0,1]^{k-1}}~ \norm y_n-\beta^{(\msout{k})}_{\vec{r}}(y_n) \norm + \norm z_n-\beta^{(\msout{k})}_{\vec{r}}(z_n) \norm \stackrel{n\to\infty}{\longrightarrow} 0
\end{equation}
and
\begin{equation} \label{eq:multiflows:yz-2}
\Sigma(b)=\lim_{n\to\infty} \ad(y_nz_ny_n^*z_n^*)(b),\quad b\in B.
\end{equation}
Let us set $Y=[(y_n)_n]$ and $Z=[(z_n)_n]$ with $Y,Z\in B_{\infty,\beta^{(\msout{k})}}^{\beta^{(\msout{k})}_\infty}$.
Moreover set $X=YZY^*Z^*$.
Note that since $\CD$ is a Kirchberg algebra, \autoref{cor:dimrokc-cor} implies that $\beta$ is equivariantly $\CO_\infty$-absorbing.
By \cite[Proposition 2.19(iii)]{Szabo18ssa2}, the unitary $X$ is thus homotopic to the unit inside $B_{\infty,\beta^{(\msout{k})}}^{\beta^{(\msout{k})}_\infty}$.
Write $X=\exp(iH_1)\cdots\exp(iH_r)$ for certain self-adjoint elements $H_1,\dots,H_r\in B_{\infty,\beta^{(\msout{k})}}^{\beta^{(\msout{k})}_\infty}$.
Set $L'=\norm H_1\norm + \dots + \norm H_r \norm$.
For $l=1,\dots,r$, represent $H_l$ via a sequence of self-adjoint elements $h_{l,n}\in B$ with $\norm h_{l,n}\norm\leq\norm H_l\norm$.
We define a sequence of continuous paths $x_n: [0,1]\to\CU(B)$ via
\[
x_n(t) = \exp(ith_{1,n})\cdots\exp(ith_{r,n}).
\]
Then each of these paths is $L'$-Lipschitz.
By slight abuse of notation we write $X: [0,1]\to \CU\big( B_{\infty,\beta^{(\msout{k})}}^{\beta^{(\msout{k})}_\infty} \big)$ for $X(t)=[(x_n(t))_n]$, which is then continuous and satisfies $X(0)=\eins$ and $X(1)=X$.
Also denote $x_n=x_n(1)$ for all $n$.

Since we have $\Sigma(b)=XbX^*$ for all $b\in B$ and $\beta$ and $\Sigma\circ\beta^{(k)}_t=\beta^{(k)}_t\circ\Sigma$, one has $\Sigma(b)=\beta^{(k)}_{\infty,t}(X)b\beta^{(k)}_{\infty,t}(X^*)$ for all $t\in\IR$.
It follows that for all $t\in\IR$, one has that the element $X\beta^{(k)}_{\infty,t}(X^*)$ commutes with all elements in $B\subset B_\infty$.

Let us for the moment fix some number $T\greater 0$.
Define $u^T_n: [0,1]\to\CU(B)$ via $u^T_n(t)=x_n(t)\beta^{(k)}_{-T}(x_n(t)^*)$.
Then $u^T_n$ is a unitary path starting at the unit and with Lipschitz constant $L\leq 2L'$.
We have
\[
\max_{0\leq t\leq 1}~\max_{\vec{r}\in [0,1]^{k-1}}~ \norm u^T_n(t)-\beta^{(\msout{k})}_{\vec{r}}(u_n^T(t)) \norm \stackrel{n\to\infty}{\longrightarrow} 0
\]
as $\beta_{-T}^{(k)}\circ\beta^{(\msout{k})}_{\vec{r}}=\beta^{(\msout{k})}_{\vec{r}}\circ\beta^{(k)}_{-T}$ and the elements $x_n(t)$ are approximately $\beta^{(\msout{k})}$-invariant by construction, and
\[
\norm[u^T_n(1),b]\norm = \norm [x_n\beta^{(k)}_{-T}(x_n^*), b] \norm \stackrel{n\to\infty}{\longrightarrow} 0 \quad\text{for all } b\in B.
\]
Due to \autoref{lem:ssa-ac-paths}, we may replace the unitary paths $u^T_n$ by ones which become approximately central along the entire path and retain all the other properties.
In other words, by changing the path $u_n$ on $(0,1)$, we may in fact assume
\[
\max_{0\leq t\leq 1}~ \norm[u^T_n(t),b]\norm \stackrel{n\to\infty}{\longrightarrow} 0 \quad\text{for all } b\in B.
\]
Let us consider the sequence of $\beta^{(k)}$-1-cocycles $\{w^{(n)}_t\}_{t\in\IR}$ given by $w^{(n)}_t = x_n\beta^{(k)}_t(x_n^*)$.
Then by what we have observed before, we have
\[
\max_{0\leq t\leq T}~ \norm [w^{(n)}_t,b] \norm \stackrel{n\to\infty}{\longrightarrow} 0,\quad b\in B,
\]
as well as
\[
\max_{0\leq t\leq T}~\max_{\vec r\in [0,1]^{k-1}}~ \norm w^{(n)}_t-\beta^{(\msout{k})}_{\vec r}(w^{(n)}_t) \norm \leq 2\cdot \max_{\vec r\in [0,1]^{k-1}}~ \norm x_n-\beta^{(\msout{k})}_{\vec r}(x_n) \norm \stackrel{n\to\infty}{\longrightarrow} 0. 
\]
This puts us into the position to apply \autoref{lem:hilfslemma}.
Given some small tolerance $\eps\greater 0$ and $\CF\fin\CD$, we can choose $T\greater 0$ and $\CG\fin\CD$ with respect to the constant $L=2L'$ and with $(B,\beta)$ in place of $(A,\alpha)$.
Without loss of generality, we choose $\CF$ in such a way that
\begin{equation} \label{eq:multiflows:F}
\Sigma(\CF)=\CF.
\end{equation}
Then the cocycles $\{w^{(n)}_t\}_{t\in\IR}$ and the unitary paths $u^T_n$ (in place of $\set{w_t}_{t\in\IR}$ and $u$ in \autoref{lem:hilfslemma}) will eventually satisfy the assumptions in \autoref{lem:hilfslemma} for large enough $n$.
By the conclusion of the statement, one finds a unitary $v_n\in\CU(B)$ such that
\begin{equation} \label{eq:multiflows:v-1}
\max_{0\leq t\leq 1} \norm w^{(n)}_t-v_n\beta^{(k)}_t(v_n^*) \norm = \max_{0\leq t\leq 1} \norm x_n\beta^{(k)}_t(x_n)^*-v_n\beta^{(k)}_t(v_n^*) \norm \leq 3\eps;
\end{equation}
\begin{equation} \label{eq:multiflows:v-2}
\max_{b\in\CF}~ \norm [v_n,b] \norm \leq 3\eps;
\end{equation}
\begin{equation} \label{eq:multiflows:v-3}
\max_{\vec{r}\in [0,1]^{k-1}}~ \norm v_n-\alpha^{(\msout{k})}_{\vec{r}}(v_n) \norm \leq 3\eps.
\end{equation}
We set $\IU_n = v_n^*x_n$, which is yet another sequence of unitaries in $B$.
Note that \eqref{eq:multiflows:v-1} translates to
\[
\max_{0\leq t\leq 1} \norm \IU_n-\beta^{(k)}_t(\IU_n) \norm \leq 3\eps.
\]
Together with \eqref{eq:multiflows:v-3} and $X\in B^{\beta^{(\msout{k})}_\infty}_{\infty,\beta^{(\msout{k})}}$ this yields
\[
\max_{\vec{r}\in [0,1]^k}~ \norm \IU_n-\beta_{\vec{r}}(\IU_n) \norm \leq 7\eps 
\]
for large enough $n$.
Finally, if we combine \eqref{eq:multiflows:yz-2}, \eqref{eq:multiflows:F} and \eqref{eq:multiflows:v-2}, we obtain
\[
\max_{b\in\CF}~ \norm \Sigma(b)-\IU_nb\IU_n^* \norm \leq 4\eps 
\]
for sufficiently large $n$.
Since $\eps\greater 0$ was an arbitrary parameter and $\CF\fin B$ was arbitrary as well, we see that the flip automorphism $\Sigma$ on $B$ is indeed approximately $\IR^k$-inner.

\textbf{Step 2:} Let $\alpha: \IR^k\curvearrowright\CD$ be an action with the Rokhlin property.
Due to the first step, $\alpha$ has approximately $\IR^k$-inner flip.
By \cite[Proposition 3.3]{Szabo18ssa}, it follows that the infinite tensor power action $\alpha^{\otimes\infty}: \IR^k\curvearrowright\CD^{\otimes\infty}$ is strongly self-absorbing.
In view of \autoref{rem:Rk-box}, we may apply \autoref{cor:dimrokc-cor} to $\alpha$ and $\alpha^{\otimes\infty}$ in place of $\gamma$, and see that
\[
\alpha\scc\alpha\otimes\alpha^{\otimes\infty} \cong \alpha^{\otimes\infty},
\]
which implies that $\alpha$ is semi-strongly self-absorbing.

\textbf{Step 3:} For $i=0,1$, let $\alpha^{(i)}: \IR^k\curvearrowright\CD$ be two actions with the Rokhlin property.
By the previous step, they are semi-strongly self-absorbing.
If we apply \autoref{cor:dimrokc-cor} to $\alpha^{(0)}$ in place of $\alpha$ and $\alpha^{(1)}$ in place of $\gamma$, then it follows that $\alpha^{(0)}\vscc\alpha^{(0)}\otimes\alpha^{(1)}$.
If we exchange the roles of $\alpha^{(0)}$ and $\alpha^{(1)}$ and repeat this argument, we conclude $\alpha^{(0)}\vscc\alpha^{(1)}$.

This finishes the induction step and the proof.
\end{proof}

We observe the following consequence as a combination of all of our main results for $\IR^k$-actions; this is new even for ordinary flows.

\begin{cor} \label{cor:multiflow-spi}
Let $A$ be a separable \cstar-algebra with $A\cong A\otimes\CO_\infty$.
Suppose that $\alpha: \IR^k\curvearrowright A$ is a multi-flow.
The following are equivalent:
\begin{enumerate}[label=\textup{(\roman*)},leftmargin=*]
\item $\alpha$ has the Rokhlin property;
\item $\alpha$ has finite Rokhlin dimension with commuting towers;
\item $\alpha\vscc\alpha\otimes\gamma$ for any multi-flow $\gamma:\IR^k\curvearrowright\CO_\infty$ with the Rokhlin property;
\item $\alpha\vscc\alpha\otimes\gamma$ for every multi-flow $\gamma:\IR^k\curvearrowright\CO_\infty$ with the Rokhlin property.
\end{enumerate}
\end{cor}
\begin{proof}
This follows directly from \autoref{thm:multiflows} and \autoref{cor:dimrokc-cor}.
\end{proof}

Once we combine \autoref{cor:multiflow-spi} and \autoref{thm:multiflows}, we obtain \autoref{Thm-C} as a direct consequence.

The following remains open:

\begin{question}
Let $\alpha: \IR^k\curvearrowright A$ be a multi-flow on a Kirchberg algebra.
Suppose that for every $\vec r\in\IR^k$, the flow on $A$ given by $t\mapsto \alpha_{t\vec r}$ has the Rokhlin property.
Does it follow that $\alpha$ has the Rokhlin property?
\end{question}


\bibliographystyle{gabor}
\bibliography{master}

\begin{thebibliography}{10}
\providecommand{\url}[1]{\texttt{#1}}
\providecommand{\urlprefix}{URL }

\bibitem{BarlakSzabo16ss}
S.~Barlak, G.~Szab{\'o}: Sequentially split $*$-homomorphisms between
  \cstar-algebras.
\newblock Int. J. Math. 27 (2016), no.~12.
\newblock 48 pages.

\bibitem{BarlakSzaboVoigt16}
S.~Barlak, G.~Szab{\'o}, C.~Voigt: The spatial {R}okhlin property for actions
  of compact quantum groups.
\newblock J. Funct. Anal. 272 (2016), no.~6, pp. 2308--2360.

\bibitem{Bartels17}
A.~Bartels: Coarse flow spaces for relatively hyperbolic groups.
\newblock Compositio Math. 153 (2017), no.~4, pp. 745--779.

\bibitem{BBSTWW}
J.~Bosa, N.~Brown, Y.~Sato, A.~Tikuisis, S.~White, W.~Winter: Covering
  dimension of {$\mathrm{C}^*$}-algebras and 2-coloured classification.
\newblock Mem. Amer. Math. Soc., to appear  (2016).
\newblock \urlprefix\url{http://arxiv.org/abs/1506.03974}.

\bibitem{BratteliEvansKishimoto95}
O.~Bratteli, D.~E. Evans, A.~Kishimoto: The {R}ohlin property for quasi-free
  automorphisms of the {F}ermion algebra.
\newblock Proc. London Math. Soc. 71 (1995), no.~3, pp. 675--694.

\bibitem{BratteliKishimoto00}
O.~Bratteli, A.~Kishimoto: Trace scaling automorphisms of certain stable {AF}
  algebras {II}.
\newblock Q. J. Math. 51 (2000), no.~2, pp. 131--154.

\bibitem{BratteliKishimotoRobinson07}
O.~Bratteli, A.~Kishimoto, D.~W. Robinson: Rohlin flows on the {C}untz algebra
  {$\CO_\infty$}.
\newblock J. Funct. Anal. 248 (2007), pp. 472--511.

\bibitem{BratteliKishimotoRordamStormer93}
O.~Bratteli, A.~Kishimoto, M.~R{\o}rdam, E.~St{\o}rmer: The crossed product of
  a {UHF} algebra by a shift.
\newblock Ergodic Theory Dynam. Systems 13 (1993), no.~4, pp. 615--626.

\bibitem{Brown00}
L.~G. Brown: Continuity of actions of groups and semigroups on {B}anach spaces.
\newblock J. London Math. Soc. 62 (2000), no.~1, pp. 107--116.

\bibitem{BrownTikuisisZelenberg17}
N.~Brown, A.~Tikuisis, A.~Zelenberg: Rokhlin dimension for
  \cstar-correspondences.
\newblock Houston J. Math., to appear  (2017).
\newblock \urlprefix\url{https://arxiv.org/abs/1608.03214}.

\bibitem{CETWW}
J.~Castillejos, S.~Evington, A.~Tikuisis, S.~White, W.~Winter: \cstar-algebras
  with property {G}amma  (2018).
\newblock In preparation.

\bibitem{ChoiEffros76}
M.-D. Choi, E.~G. Effros: The completely positive lifting problem for
  \cstar-algebras.
\newblock Ann. of Math. 104 (1976), no.~3, pp. 585--609.

\bibitem{Connes75}
A.~Connes: Outer conjugacy classes of automorphisms of factors.
\newblock Ann. Sci. {\'E}cole Norm. Sup. 8 (1975), pp. 383--419.

\bibitem{Connes76}
A.~Connes: Classification of injective factors. {C}ases {II${}_1$,
  II${}_\infty$, III${}_\lambda$, $\lambda\neq1$}.
\newblock Ann. of Math. 74 (1976), pp. 73--115.

\bibitem{Connes77}
A.~Connes: Periodic automorphisms of the hyperfinite factors of type {II}$_1$.
\newblock Acta Sci. Math. 39 (1977), pp. 39--66.

\bibitem{DadarlatWinter09}
M.~Dadarlat, W.~Winter: On the {$KK$}-theory of strongly self-absorbing
  \cstar-algebras.
\newblock Math. Scand. 104 (2009), no.~1, pp. 95--107.

\bibitem{DelabieTointon17}
T.~Delabie, M.~Tointon: The asymptotic dimension of box spaces of virtually
  nilpotent groups.
\newblock Discrete Math. 341 (2018), no.~4, pp. 1036--1040.

\bibitem{ElliottEvansKishimoto98}
G.~A. Elliott, D.~E. Evans, A.~Kishimoto: Outer conjugacy classes of trace
  scaling automorphisms of stable {UHF} algebras.
\newblock Math. Scand. 83 (1988), no.~1, pp. 74--86.

\bibitem{ElliottToms08}
G.~A. Elliott, A.~S. Toms: Regularity properties in the classification program
  for separable amenable \cstar-algebras.
\newblock Bull. Amer. Math. Soc. 45 (2008), pp. 229--245.

\bibitem{EmersonGreenleaf67}
W.~R. Emerson, F.~P. Greenleaf: Covering properties and {F}{\o}lner conditions
  for locally compact groups.
\newblock Math. Zeitschr. 102 (1967), pp. 370--384.

\bibitem{EvansKishimoto97}
D.~E. Evans, A.~Kishimoto: Trace scaling automorphisms of certain stable {AF}
  algebras.
\newblock Hokkaido Math. J. 26 (1997), pp. 211--224.

\bibitem{Gardella14_1}
E.~Gardella: Classification theorems for circle actions on {K}irchberg
  algebras, {I}  (2014).
\newblock \urlprefix\url{http://arxiv.org/abs/1405.2469}.

\bibitem{Gardella14_2}
E.~Gardella: Classification theorems for circle actions on {K}irchberg
  algebras, {II}  (2014).
\newblock \urlprefix\url{http://arxiv.org/abs/1406.1208}.

\bibitem{Gardella17}
E.~Gardella: Compact group actions with the {R}okhlin property and their
  crossed products.
\newblock J. Noncomm. Geom. 11 (2017), no.~4, pp. 1593--1626.

\bibitem{GardellaHirshbergSantiago17}
E.~Gardella, I.~Hirshberg, L.~Santiago: Rokhlin dimension:\ duality, tracial
  properties, and crossed products  (2017).
\newblock \urlprefix\url{https://arxiv.org/abs/1709.00222v1}.

\bibitem{GardellaKalantarLupini17}
E.~Gardella, M.~Kalantar, M.~Lupini: Rokhlin dimension for compact quantum
  group actions.
\newblock J. Noncomm. Geom., to appear  (2017).
\newblock \urlprefix\url{https://arxiv.org/abs/1703.10999}.

\bibitem{GardellaLupini16}
E.~Gardella, M.~Lupini: Applications of model theory to \cstar-dynamics.
\newblock J. Funct. Anal. 275 (2018), no.~7, pp. 1889--1942.

\bibitem{GardellaSantiago16}
E.~Gardella, L.~Santiago: Equivariant {$*$}-homomorphisms, {R}okhlin contraints
  and equivariant {UHF}-absorption.
\newblock J. Funct. Anal. 270 (2016), no.~7, pp. 2543--2590.

\bibitem{GoldsteinIzumi11}
P.~Goldstein, M.~Izumi: Quasi-free actions of finite groups on the {C}untz
  algebra {$\CO_\infty$}.
\newblock Tohoku Math. J. 63 (2011), pp. 729--749.

\bibitem{HermanJones82}
R.~Herman, V.~Jones: Period two automorphisms of {UHF} \cstar-algebras.
\newblock J. Funct. Anal. 45 (1982), no.~2, pp. 169--176.

\bibitem{HermanOcneanu84}
R.~Herman, A.~Ocneanu: Stability for integer actions on {UHF} \cstar-algebras.
\newblock J. Funct. Anal. 59 (1984), pp. 132--144.

\bibitem{HirshbergPhillips15}
I.~Hirshberg, N.~C. Phillips: Rokhlin dimension: obstructions and permanence
  properties.
\newblock Doc. Math. 20 (2015), pp. 199--236.

\bibitem{HirshbergRordamWinter07}
I.~Hirshberg, M.~R{\o}rdam, W.~Winter: \cstar-algebras, stability and strongly
  self-absorbing \cstar-algebras.
\newblock Math. Ann. 339 (2007), no.~3, pp. 695--732.

\bibitem{HirshbergSzaboWinterWu17}
I.~Hirshberg, G.~Szab{\'o}, W.~Winter, J.~Wu: Rokhlin dimension for flows.
\newblock Comm. Math. Phys. 353 (2017), no.~1, pp. 253--316.

\bibitem{HirshbergWinter07}
I.~Hirshberg, W.~Winter: Rokhlin actions and self-absorbing \cstar-algebras.
\newblock Pacific J. Math. 233 (2007), no.~1, pp. 125--143.

\bibitem{HirshbergWinterZacharias15}
I.~Hirshberg, W.~Winter, J.~Zacharias: Rokhlin dimension and \cstar-dynamics.
\newblock Comm. Math. Phys. 335 (2015), pp. 637--670.

\bibitem{Izumi04}
M.~Izumi: Finite group actions on \cstar-algebras with the {R}ohlin property
  {I}.
\newblock Duke Math. J. 122 (2004), no.~2, pp. 233--280.

\bibitem{Izumi04II}
M.~Izumi: Finite group actions on \cstar-algebras with the {R}ohlin property
  {II}.
\newblock Adv. Math. 184 (2004), no.~1, pp. 119--160.

\bibitem{Izumi12OWR}
M.~Izumi: Poly-{$\IZ$} group actions on {K}irchberg algebras.
\newblock Oberwolfach Rep. 9 (2012), pp. 3170--3173.

\bibitem{IzumiMatui10}
M.~Izumi, H.~Matui: {$\IZ^2$}-actions on {K}irchberg algebras.
\newblock Adv. Math. 224 (2010), pp. 355--400.

\bibitem{JiangSu99}
X.~Jiang, H.~Su: On a simple unital projectionless \cstar-algebra.
\newblock Amer. J. Math. 121 (1999), no.~2, pp. 359--413.

\bibitem{Jones80}
V.~F.~R. Jones: Actions of finite groups on the hyperfinite type
  {$\mathrm{II}_1$} factor.
\newblock Mem. Amer. Math. Soc. 28 (1980), no. 237.
\newblock 70 pages.

\bibitem{KatayamaSutherlandTakesaki98}
Y.~Katayama, C.~E. Sutherland, M.~Takesaki: The characteristic square of a
  factor and the cocycle conjugacy of discrete group actions on factors.
\newblock Invent. Math. 132 (1998), pp. 331--380.

\bibitem{KatsuraMatui08}
T.~Katsura, H.~Matui: Classification of uniformly outer actions of {$\IZ^2$} on
  {UHF} algebras.
\newblock Adv. Math 218 (2008), pp. 940--968.

\bibitem{KawahigashiSutherlandTakesaki92}
Y.~Kawahigashi, C.~Sutherland, M.~Takesaki: The structure of the automorphism
  group of an injective factor and the cocycle conjugacy of discrete abelian
  group actions.
\newblock Acta Math. 169 (1992), pp. 105--130.

\bibitem{Khukhro12}
A.~Khukhro: Box spaces, group extensions and coarse embeddings into {H}ilbert
  space.
\newblock J. Funct. Anal. 263 (2012), no.~1, pp. 115--128.

\bibitem{KirchbergC}
E.~Kirchberg: The {C}lassification of {P}urely {I}nfinite \cstar-{A}lgebras
  {U}sing {K}asparov's {T}heory (2003).
\newblock Preprint.

\bibitem{Kirchberg04}
E.~Kirchberg: Central sequences in \cstar-algebras and strongly purely infinite
  algebras.
\newblock Operator Algebras: The Abel Symposium 1 (2004), pp. 175--231.

\bibitem{KirchbergPhillips00}
E.~Kirchberg, N.~C. Phillips: Embedding of exact \cstar-algebras in the {C}untz
  algebra {$\CO_2$}.
\newblock J. reine angew. Math. 525 (2000), pp. 17--53.

\bibitem{Kishimoto95}
A.~Kishimoto: The {R}ohlin property for automorphisms of {UHF} algebras.
\newblock J. reine angew. Math. 465 (1995), pp. 183--196.

\bibitem{Kishimoto96_R}
A.~Kishimoto: A {R}ohlin property for one-parameter automorphism groups.
\newblock Comm. Math. Phys. 179 (1996), no.~3, pp. 599--622.

\bibitem{Kishimoto96}
A.~Kishimoto: The {R}ohlin property for shifts on {UHF} algebras and
  automorphisms of {C}untz algebras.
\newblock J. Funct. Anal. 140 (1996), pp. 100--123.

\bibitem{Kishimoto98}
A.~Kishimoto: Automorphisms of {AT} algebras with the {R}ohlin property.
\newblock J. Operator Theory 40 (1998), pp. 277--294.

\bibitem{Kishimoto98II}
A.~Kishimoto: Unbounded derivations in {AT} algebras.
\newblock J. Funct. Anal. 160 (1998), pp. 270--311.

\bibitem{Kishimoto01}
A.~Kishimoto: {UHF} flows and the flip automorphism.
\newblock Rev. Math. Phys. 13 (2001), no.~9, pp. 1163--1181.

\bibitem{Kishimoto02}
A.~Kishimoto: Rohlin flows on the {C}untz algebra {$\CO_2$}.
\newblock Int. J. Math. 13 (2002), no.~10, pp. 1065--1094.

\bibitem{Liao16}
H.-C. Liao: A {R}okhlin type theorem for simple \cstar-algebras of finite
  nuclear dimension.
\newblock J. Funct. Anal. 270 (2016), no.~10, pp. 3675--3708.

\bibitem{Liao17}
H.-C. Liao: Rokhlin dimension of {$\IZ^m$}-actions on simple \cstar-algebras.
\newblock Int. J. Math. 28 (2017), no.~7.
\newblock 22 pages.

\bibitem{Masuda07}
T.~Masuda: Evans--{K}ishimoto type argument for actions of discrete amenable
  groups on {McDuff} factors.
\newblock Math. Scand. 101 (2007), pp. 48--64.

\bibitem{Matui08}
H.~Matui: Classification of outer actions of {$\IZ^N$} on {$\CO_2$}.
\newblock Adv. Math. 217 (2008), pp. 2872--2896.

\bibitem{Matui10}
H.~Matui: {$\IZ$}-actions on {AH} algebras and {$\IZ^2$}-actions on {AF}
  algebras,.
\newblock Comm. Math. Phys. 297 (2010), pp. 529--551.

\bibitem{Matui11}
H.~Matui: {$\IZ^N$}-actions on {UHF} algebras of infinite type.
\newblock J. reine angew. Math 657 (2011), pp. 225--244.

\bibitem{MatuiSato12_2}
H.~Matui, Y.~Sato: {$\CZ$}-stability of crossed products by strongly outer
  actions.
\newblock Comm. Math. Phys. 314 (2012), no.~1, pp. 193--228.

\bibitem{MatuiSato12}
H.~Matui, Y.~Sato: Strict comparison and {$\mathcal{Z}$}-absorption of nuclear
  \cstar-algebras.
\newblock Acta Math. 209 (2012), no.~1, pp. 179--196.

\bibitem{MatuiSato14}
H.~Matui, Y.~Sato: {$\CZ$}-stability of crossed products by strongly outer
  actions {II}.
\newblock Amer. J. Math. 136 (2014), pp. 1441--1497.

\bibitem{MatuiSato14UHF}
H.~Matui, Y.~Sato: Decomposition rank of {UHF}-absorbing \cstar-algebras.
\newblock Duke Math. J. 163 (2014), no.~14, pp. 2687--2708.

\bibitem{Nakamura99}
H.~Nakamura: The {R}ohlin property for {$\IZ^2$}-actions on {UHF} algebras.
\newblock J. Math. Soc. Japan 51 (1999), no.~3, pp. 583--612.

\bibitem{Nakamura00}
H.~Nakamura: Aperiodic automorphisms of nuclear purely infinite simple
  \cstar-algebras.
\newblock Ergodic Theory Dynam. Systems 20 (2000), pp. 1749--1765.

\bibitem{NowakYuLSG}
P.~W. Nowak, G.~Yu: Large {S}cale {G}eometry.
\newblock EMS (2012).

\bibitem{Ocneanu85}
A.~Ocneanu: Actions of discrete amenable groups on von {N}eumann algebras,
  \emph{Lecture Notes in Mathematics}, Vol. 1138.
\newblock Springer-Verlag, Berlin (1985).

\bibitem{PackerRaeburn89}
J.~A. Packer, I.~Raeburn: Twisted crossed products of \cstar-algebras.
\newblock Math. Proc. Cambridge Philos. Soc. 106 (1989), no.~2, pp. 293--311.

\bibitem{Phillips00}
N.~C. Phillips: A classification theorem for nuclear purely infinite simple
  \cstar-algebras.
\newblock Doc. Math. 5 (2000), pp. 49--114.

\bibitem{RoeCG}
J.~Roe: Lectures on {C}oarse {G}eometry, \emph{University Lecture Series},
  Vol.~31.
\newblock AMS (2003).

\bibitem{Rordam04}
M.~R{\o}rdam: The stable and the real rank of {$\CZ$}-absorbing
  \cstar-algebras.
\newblock Int. J. Math. 15 (2004), pp. 1065--1084.

\bibitem{Santiago15}
L.~Santiago: Crossed products by actions of finite groups with the {R}okhlin
  property.
\newblock Int. J. Math. 26 (2015).
\newblock 31 pages.

\bibitem{SatoWhiteWinter15}
Y.~Sato, S.~White, W.~Winter: Nuclear dimension and {$\CZ$}-stability.
\newblock Invent. Math. 202 (2015), pp. 893--921.

\bibitem{SutherlandTakesaki89}
C.~E. Sutherland, M.~Takesaki: Actions of discrete amenable groups on injective
  factors of type {$\mathrm{III}_\lambda,~\lambda\neq 1$}.
\newblock Pacific J. Math. 137 (1989), pp. 405--444.

\bibitem{Szabo15plms}
G.~Szab{\'o}: The {R}okhlin dimension of topological {$\IZ^m$}-actions.
\newblock Proc. Lond. Math. Soc. 110 (2015), no.~3, pp. 673--694.

\bibitem{Szabo17R}
G.~Szab{\'o}: The classification of {R}okhlin flows on \cstar-algebras  (2017).
\newblock \urlprefix\url{http://arxiv.org/abs/1706.09276}.

\bibitem{Szabo17ssa3}
G.~Szab{\'o}: Strongly self-absorbing \cstar-dynamical systems, {III}.
\newblock Adv. Math. 316 (2017), pp. 356--380.

\bibitem{Szabo18kp}
G.~Szab{\'o}: Equivariant {K}irchberg--{P}hillips-type absorption for amenable
  group actions.
\newblock Comm. Math. Phys. 361 (2018), no.~3, pp. 1115--1154.

\bibitem{Szabo18ssa}
G.~Szab{\'o}: Strongly self-absorbing \cstar-dynamical systems.
\newblock Trans. Amer. Math. Soc. 370 (2018), pp. 99--130.

\bibitem{Szabo18ssa2}
G.~Szab{\'o}: Strongly self-absorbing \cstar-dynamical systems, {II}.
\newblock J. Noncommut. Geom. 12 (2018), no.~1, pp. 369--406.

\bibitem{SzaboWuZacharias17}
G.~Szab{\'o}, J.~Wu, J.~Zacharias: Rokhlin dimension for actions of residually
  finite groups.
\newblock Ergodic Theory Dynam. Systems, to appear  (2017).
\newblock \urlprefix\url{http://dx.doi.org/10.1017/etds.2017.113}.

\bibitem{TikuisisWhiteWinter17}
A.~Tikuisis, S.~White, W.~Winter: Quasidiagonality of nuclear \cstar-algebras.
\newblock Ann. Math. 185 (2017), pp. 229--284.

\bibitem{TomsWinter07}
A.~S. Toms, W.~Winter: Strongly self-absorbing \cstar-algebras.
\newblock Trans. Amer. Math. Soc. 359 (2007), no.~8, pp. 3999--4029.

\bibitem{Winter10}
W.~Winter: Decomposition rank and {$\mathcal{Z}$}-stability.
\newblock Invent. Math. 179 (2010), no.~2, pp. 229--301.

\bibitem{Winter11}
W.~Winter: Strongly self-absorbing \cstar-algebras are {$\CZ$}-stable.
\newblock J. Noncomm. Geom. 5 (2011), no.~2, pp. 253--264.

\bibitem{Winter12}
W.~Winter: Nuclear dimension and {$\mathcal{Z}$}-stability of pure
  \cstar-algebras.
\newblock Invent. Math. 187 (2012), no.~2, pp. 259--342.

\bibitem{Winter14Lin}
W.~Winter: Localizing the {E}lliott conjecture at strongly self-absorbing
  \cstar-algebras, with an appendix by {H.} {L}in.
\newblock J. reine angew. Math. 692 (2014), pp. 193--231.

\bibitem{WinterZacharias09}
W.~Winter, J.~Zacharias: Completely positive maps of order zero.
\newblock M{\"u}nster J. Math. 2 (2009), pp. 311--324.

\bibitem{WinterZacharias10}
W.~Winter, J.~Zacharias: The nuclear dimension of \cstar-algebras.
\newblock Adv. Math. 224 (2010), no.~2, pp. 461--498.

\end{thebibliography}

\end{document}